\documentclass[letterpaper,journal]{IEEEtran} 
\usepackage{balance}
\usepackage{microtype}
\usepackage{graphicx}
\usepackage{subfigure}
\usepackage{booktabs} 

\usepackage[hidelinks]{hyperref}
\usepackage{threeparttable}
\usepackage{pifont}

\usepackage{multirow}   

\usepackage{algorithm}
\usepackage{algorithmicx}
\usepackage{algpseudocode}

\usepackage{caption}

\usepackage{cite}

\usepackage{subcaption}

\usepackage{amsmath}
\usepackage{amssymb}
\usepackage{mathtools}
\usepackage{amsthm}
\usepackage{enumitem}

\usepackage[capitalize,noabbrev]{cleveref}
\crefname{equation}{Eq.}{Eqs.}
\Crefname{equation}{Eq.}{Eqs.}

\usepackage{xcolor}            
\usepackage{mdframed}

\crefname{assumption}{Assumption}{Assumptions}

\theoremstyle{plain}
\newtheorem{theorem}{Theorem}

\newtheorem{lemma}{Lemma}

\theoremstyle{definition}
\newtheorem{definition}{Definition}
\newtheorem{assumption}{Assumption}
\theoremstyle{remark}
\newtheorem{remark}{Remark}

\newcommand{\CenteredUppercase}[1]{%
  \begin{center}
    \MakeUppercase{#1}
  \end{center}
}

\makeatletter

\begin{document}

\title{Distributed Online Randomized Gradient-Free Optimization with Compressed Communication}
\author{Longkang~Zhu, Xinli~Shi, \IEEEmembership{Senior~Member, IEEE},~Xiangping~Xu, 
Jinde~Cao, \IEEEmembership{Fellow, IEEE}, and Xiangyong~Chen 
 \thanks{This work was supported by the National Natural Science Foundation of China under Grant Nos. 62473098 and 62203108;  in part by the Australian Research Council under Grant
DE250100961; in part by China Postdoctoral Science Foundation under Grant Nos. 2022M720720 and 2023T160105; and in part by Southeast University
Interdisciplinary Research Program for Young Scholars. (Corresponding author: Xinli Shi)}
 \thanks{L. Zhu is with the School of Cyber Science and Engineering, Southeast University, Nanjing 210096, China (e-mail: 230248643@seu.edu.cn).}
\thanks{X. Shi is with the School of Cyber Science and Engineering, Southeast University, Nanjing, China, 210096, and also with the School of
Engineering, RMIT University, Melbourne, VIC 3001, Australia. (e-mail: xinli\_shi@seu.edu.cn).}
 \thanks{X.~Xu and J. Cao are with the School of Mathematics, Southeast University, Nanjing 210096, China (e-mail: xpxu2021@seu.edu.cn; jdcao@seu.edu.cn).}
 \thanks{X. Chen is with the School of Automation and
Electrical Engineering, Linyi University, Linyi, 276005, China (e-mail: xy8305@163.com).}
 }

\IEEEtitleabstractindextext{
\begin{abstract}
This paper addresses two fundamental challenges in distributed online convex optimization: communication efficiency and optimization under limited feedback. We propose a unified framework named Online Compressed Gradient Tracking (\textbf{OCGT}), which includes two variants: One-point Bandit Feedback (\textbf{OCGT-BF}) and Stochastic Gradient Feedback (\textbf{OCSGT}). The proposed algorithms harness data compression and either gradient-free or stochastic gradient optimization techniques within distributed networks. The proposed framework incorporates a compression scheme with error compensation mechanisms to reduce communication overhead while maintaining convergence guarantees. Unlike traditional approaches that assume perfect communication and full gradient access, \textbf{OCGT} operates effectively under practical constraints by combining gradient-like tracking with one-point or stochastic gradient feedback estimation. We provide a theoretical analysis demonstrating dynamic regret bounds for both variants. Finally, extensive experiments validate that \textbf{OCGT} achieves low dynamic regret while significantly reducing communication requirements.
\end{abstract}

\begin{IEEEkeywords}
Distributed online optimization, gradient tracking, bandit feedback, communication compression, dynamic regret
\end{IEEEkeywords}}

\maketitle
\IEEEdisplaynontitleabstractindextext

\IEEEpeerreviewmaketitle

\section{Introduction}\label{sec:introduction}

In recent years, the rapid growth of interconnected systems and data-driven applications has spurred interest in scalable, real-time optimization methods suitable for decentralized environments.
Distributed online convex optimization (DOCO) is a dynamic framework for addressing optimization challenges over a network of agents where decisions need to be made sequentially over time, often under uncertainty. This field has received significant attention due to its applications in machine learning, target tracking, and resource management (see, e.g., \cite{rabbat2004distributed, 9953189, shahrampour2017distributed}).

Numerous online and offline optimization algorithms have been introduced to handle large-scale problems in distributed settings over the past few years. In \cite{nedic2009distributed},  a distributed subgradient descent method (DGD) is introduced for finite-sum optimization, typically using diminishing stepsizes. Subsequent improvements, such as EXTRA and DGT \cite{shi2015extra,qu2017harnessing}, improve convergence by employing gradient tracking techniques, which use auxiliary variables to estimate the global gradient. The works \cite{yan2012distributed} explore distributed convex optimization in online environments, providing algorithms demonstrating $\mathcal{O}(\sqrt{T})$ static regret for convex loss functions. Furthermore, \cite{shahrampour2017distributed} introduces a distributed online mirror descent algorithm and defined dynamic regret as a performance metric. Distributed online optimization algorithms based on gradient tracking are studied in \cite{zhang2019distributed, carnevale2022gtadam}.

In distributed optimization, communication efficiency is crucial as data scale and complexity continue to grow. Transmitting high-dimensional vectors among agents can consume substantial bandwidth and energy, particularly for resource-constrained devices or networks. Driven by applications in sensor networks and large-scale machine learning, extensive research focuses on developing distributed online and offline optimization algorithms that operate under communication constraints \cite{10366852, cao2021decentralized, yuan2022distributed}.

In offline settings, various strategies aim to reduce communication overhead through compression. For example, gradient sparsification techniques \cite{stich2018sparsified} enable nodes to transmit only a subset of gradient components (e.g., those with the largest magnitudes) while locally storing the rest. Error feedback methods \cite{seide20141, richtarik2021ef21} effectively address convergence issues caused by biased compression operators. Research also explores adaptive compression, which dynamically adjusts compression levels to improve convergence \cite{zhang2023quantized, 10423571}. These efforts primarily target offline optimization. Meanwhile, compressed communication in DOCO is gaining interest. For instance, \cite{cao2023decentralized} introduces a decentralized method with compressed communication that achieves optimal static regret bounds for convex objectives, while \cite{yuan2022distributed} examines the effects of random quantization in bandit settings. Similarly, \cite{11165425} investigates quantized communication combined with one-point estimators for distributed online optimization in a zero-sum game between two networks.  Additionally, \cite{tu2022distributed} provides theoretical guarantees for compressed communication under both full-information and one-point feedback for strongly convex functions. Building on this, \cite{zhang2023quantized} proposes a quantized distributed online Frank-Wolfe algorithm, and \cite{liu2025decentralized} employs a time-varying scaling compression method to mitigate communication bottlenecks.

Another significant challenge in distributed optimization is the limited accessibility of gradient information. In many practical scenarios, such as multi-agent robotics \cite{jiang2023black} or legacy black-box systems \cite{nesterov2017random}, obtaining full gradient information can be infeasible or prohibitively expensive. Consequently, agents often only have access to function value evaluations (zeroth-order information). This has led to the development of gradient-free algorithms, enabling each agent to rely solely on function values obtained at queried points. Recently, several works have explored distributed online bandit optimization and distributed offline gradient-free optimization \cite{yi2020distributed, yuan2020distributed, wan2019Subgradient}. Building on the foundational literature for centralized gradient-free optimization \cite{nesterov2017random}, the studies in \cite{yuan2015randomized} and \cite{pang2019randomized} introduce distributed gradient-free algorithms based on two-point feedback.

\begin{table*}[ht]
\centering
\caption{Comparison with existing decentralized online optimization methods. The notation $\mathcal{O}(\cdot)$ represents the asymptotic upper bound. $T$ is the total number of iterations. $D_1, D_2, D_3, D_4$ are constants depending on initial states, graph topology, and loss function properties. $\mathrm{U}_0$ represents the initial error state vector (defined before Theorem \ref{Th1}). }
\label{tab:comparison}
\begin{threeparttable}
\resizebox{\textwidth}{!}{
\begin{tabular}{lcccc}
\toprule
\textbf{Reference} & \textbf{Compression?} & \textbf{Gradient type} & \textbf{Regret Type} & \textbf{Regret Bound} \\
\midrule
\cite{cao2023decentralized} & \ding{51} & Local full gradient & Static  & $\mathcal{O}\bigl(\sqrt{T}\bigr)$ \\
\cite{liu2025decentralized} & \ding{51} & Local full gradient & Static  & $\mathcal{O}(\log (T))$ \\
\cite{yuan2022distributed} & \ding{51} & One-point bandit & Static  & $\mathcal{O}\bigl(T^{3/4}\bigr)$ \\

\cite{tu2022distributed} & \ding{51} & One-point bandit & Static  & $\mathcal{O}\bigl(T^{2/3}\ln^{1/3}(T)\bigr)$ \\

\cite{wang2022pushsum}& \ding{55}& One-point bandit & Static&$\mathcal{O}(T^{2 / 3} \ln ^{2 / 3}(T))$ \\
\cite{wang2024distributed}& \ding{55}& One-point bandit & Dynamic& $\mathcal{O}\left(T^{(3 / 4+\rho / 4+\kappa)}\right)\textsuperscript{a}$\\
\cite{zhang2023quantized} & \ding{51} & Local full gradient & Dynamic  & $\mathcal{O}\Bigl(D_{1}+D_{2}\,T+D_{3}\sum_{t=1}^{T}\epsilon_{d, k_t}+\mathcal{H}_{T}+\mathcal{P}_{T}\Bigr)$\textsuperscript{b} \\
\cite{hou2024dynamic} & \ding{55} & Local full gradient & Dynamic  & $\mathcal{O}\bigl(\sqrt{T}\,\bigl(\mathcal{C}_{T}+1\bigr)\bigr)$\textsuperscript{c} \\
\cite{zhang2019distributed} & \ding{55} & Local full gradient & Dynamic  & $\mathcal{O}\bigl(D_4+\mathcal{P}_{T}+\mathcal{V}_{T}\bigr)$\textsuperscript{d} \\

\addlinespace[0.3em]
\midrule
\addlinespace[0.5em]
\multirow{2}{*}{\textbf{Our Work}}
& \ding{51}
& One-point bandit
& Dynamic (Theorem \ref{Th1})
& $\mathcal{O}( \|\mathrm{U}_0\| +  \sum_{t=0}^{T-1} \mathbf{p}_t^2 +  \sum_{t=0}^{T-1}\mathbf{v}_t^2  +\sum_{t=0}^{T-1} v_t^2  +  \sum_{t=0}^{T-1} \bar{\sigma}_t^2 )\textsuperscript{e}$  \\
& \ding{51}& Stochastic gradient & Dynamic (Theorem \ref{Th3}) 
 & $\mathcal{O}( \|\mathrm{U}_0\| +  \sum_{t=0}^{T-1} \mathbf{p}_t^2 +  \sum_{t=0}^{T-1}\mathbf{v}_t^2 +  \sum_{t=0}^{T-1} \hat{\sigma}_t^2 )$ \\ 
\bottomrule
\end{tabular}
}
\begin{tablenotes}
\footnotesize
\item[a] The tunable parameters $\rho$ and $\kappa$ satisfy $\rho, \kappa \in(0,1)$;
\item[b] Cumulative function variation: $\mathcal{H}_t := \sum _{t=1}^{T-1}\max _{i \in \mathcal{V}} \max _{\boldsymbol{x} \in \mathbb{R}^n}\left|f_{i, t+1}(x)-f_{i, t}(x)\right|$;

Cumulative gradient variation : $\mathcal{P}_t := \sum _{t=1}^{T-1}\sup _{i\in \mathcal{V}} \sup _{x \in \mathbb{R}^n}\left\|\nabla f_{i,t+1}(x)-\nabla f_{i,t}(x)\right\|_\infty$;

$\epsilon_{d, k_t}$ denotes a quantization resolution that is dependent on the quantization levels $k_t$ and the dimension $d$;

\item[c] The path length of an arbitrary sequence:
$
\mathcal{C}_T=\sum_{t=1}^T\left\|u_{t+1}-u_t\right\|
$;
\item[d] Cumulative drift of optimal points: $\mathcal{V}_t :=\sum_{t=0}^{T-1}\left\|x_{t+1}^*-x_t^*\right\|$; 
\item[e] $\mathbf{p}_t$ and $\mathbf{v}_t$ quantify gradient variation and optimal point drift (defined after Definition \ref{def:dynamic_regret}); $v_t$ is the smoothing parameter for bandit feedback;

$\bar{\sigma}_t^2$ measures gradient heterogeneity at the optimum (Assumption \ref{assump:heterogeneity});
\end{tablenotes}
\end{threeparttable}
\end{table*}

To evaluate performance metrics, online convex optimization commonly employs static regret or dynamic regret, with most of the aforementioned works using static regret. Recent works focus on dynamic regret bounds under various conditions. For example, \cite{hou2024dynamic} investigates distributed online composite optimization and establishes dynamic regret bounds, while \cite{cao2024distributed} proposes an adaptive quantization method achieving $\mathcal{O}(\sqrt{T})$ dynamic regret for both full-information and one-point feedback scenarios in a centralized setting. In \cite{zhang2019distributed}, the authors first investigate gradient tracking for online optimization, establishing a dynamic regret bound of $\mathcal{O}\bigl(D_4 + \mathcal{P}_{T} + \mathcal{V}_{T}\bigr)$. Then, \cite{alghunaim2024enhanced} and \cite{carnevale2022gtadam} refine these approaches by improving both algorithmic designs and regret analysis strategies.

In most existing DOCO works \cite{yuan2018adaptive,lu2021privacy,10551450}, it is commonly assumed that agents communicate exact information and have access to full local gradients. While recent work like \cite{10294106} relaxes these assumptions via compressed bandit algorithms for unbalanced digraphs, it focuses on static regret bounds for convex objectives. However, static regret relies on a fixed benchmark and is insufficient for characterizing algorithmic performance in dynamic environments where the optimal solution varies over time. Consequently, there remains a lack of research that investigates dynamic regret bounds under the simultaneous constraints of quantized communication and limited bandit feedback.

In this paper, we tackle two fundamental challenges in DOCO: communication efficiency and optimization under limited feedback. Our goal is to develop novel algorithms that remain effective despite communication bottlenecks and incomplete feedback. The main innovations of our work can be summarized as follows:
\begin{itemize}
\item We propose a unified online compressed gradient tracking (\textbf{OCGT}) framework to simultaneously reduce communication overhead and handle limited gradient information in large-scale distributed settings. Specifically, the proposed framework employs a tailored compression scheme and a robust error compensation mechanism, where compression reduces communication cost, while error compensation mitigates error accumulation from repeated compression. Our compression approach builds upon techniques from \cite{liao2022compressed, zhao2024faster}. Unlike methods requiring full gradients, \textbf{OCGT} integrates gradient tracking with either one-point bandit feedback (\textbf{OCGT-BF}) or stochastic gradient feedback (\textbf{OCSGT}), effectively addressing the intertwined challenges of constrained communication and incomplete feedback, ensuring robust convergence despite practical limitations. The proposed methods are compared with related works in \cref{tab:comparison}.

\item A comprehensive dynamic regret analysis is provided, demonstrating the performance of the proposed algorithms to adapt to time-varying objectives under both bandit and stochastic gradient scenarios. Specifically, we show that the proposed unified framework can effectively handle evolving objectives under limited gradient information in an online optimization framework. Our theoretical results rigorously bound the algorithms’ performance in tracking the sequence of time-varying optima. 

\item The effectiveness of \textbf{OCGT} is validated through extensive experiments on real-world data. In particular, simulations on logistic regression tasks using real datasets demonstrate the robustness and adaptability of the proposed algorithms, highlighting their applicability to practical machine learning tasks.
\end{itemize}

The remainder of this work is organized as follows. Section~\ref{PF} introduces the problem formulation, key assumptions, and relevant background on randomized gradient-free optimization and communication compression. The proposed distributed \textbf{OCGT} framework, along with the corresponding dynamic regret analysis, is presented in Section~\ref{ocgt}. A numerical experiment is provided in Section~\ref{experiments}. Finally, conclusions are summarized in Section~\ref{conclusions}.

We clarify the notation used in the subsequent analysis. Let $X_t = \left[x_1, \dots, x_n\right]^\top \in \mathbb{R}^{n \times m}$ and $\nabla F(X_t) = \left[\nabla f_1(x_1), \dots, \nabla f_n(x_n)\right]^\top \in \mathbb{R}^{n \times m}$ denote the collections of local decision variables and their corresponding gradients, respectively. The notation $\|\cdot\|$ represents the Euclidean norm for vectors, the Frobenius norm for general matrices, and the spectral norm when applied to square matrices. The symbol $\preceq$ indicates an element-wise inequality between vectors or matrices. The inner product between vectors is denoted by $\langle \cdot, \cdot \rangle$, and $\mathbf{1}_n$ denotes the $n$-dimensional column vector with all entries equal to one.

\section{Preliminaries}
\label{PF}

This section presents the problem formulation for distributed online optimization over networks of interconnected agents. We give key assumptions about loss functions and define dynamic regret as a performance metric. The proposed algorithm introduces compression methods and explores random gradient-free optimization techniques, specifically the one-point bandit method for estimating gradients in complex computational settings.

\subsection{Problem Formulation}
Consider a distributed online optimization problem over a network of $n$ agents: \begin{align}
    &\min_{x \in \mathbb{R}^m} f_t(x) \triangleq \frac{1}{n} \sum_{i=1}^n f_{i,t}(x), \quad t \geq 0, \label{P}
\end{align}
where $x \in \mathbb{R}^m$ denotes the  decision variable, and $f_{i,t}: \mathbb{R}^m \rightarrow \mathbb{R}$ is the local convex loss function of agent $i$ at iteration $t$. Each agent only has access to the evaluation of its local loss function per round and can only communicate compressed information with its neighbors under a resource-limited network.

In this work, all agents are interconnected via an undirected graph $\mathcal{G}=(\mathcal{V}, \mathcal{E})$,  where $\mathcal{V}=\{1,2, \ldots, n\}$ denotes the collection of nodes (agents), and $\mathcal{E} \subseteq \mathcal{V} \times \mathcal{V}$ is the corresponding edge collection.
For an arbitrary node $i \in \mathcal{V}$, its neighborhood set  is characterized as  $\mathcal{N}_i = \{j : (i,j) \in \mathcal{E}\} \cup \{i\}$, which signifies that node $i$ exclusively shares its local state with nodes $j \in \mathcal{N}_i$.
An undirected graph is connected if a path exists between any two distinct nodes. Besides, information exchange among nodes is facilitated through a mixing matrix $\mathbf{W} = [w_{ij}]\in[0,1]^{n\times n}$. Specifically,  a weight $w_{ij}>0$ is assigned to each edge $(i,j) \in \mathcal{E}$, indicating a direct interconnection, while $w_{ij}=0$ for node pairs $(i,j) \notin \mathcal{E}$, thereby denoting the absence of a direct link.
Moreover, the matrix $\mathbf{W}$ is subject to a specific assumption.
\begin{assumption}
\label{graph}
The graph $\mathcal{G}=(\mathcal{V}, \mathcal{E})$ is connected. Its mixing matrix $\mathbf{W}$ is symmetric and double stochastic, i.e., $\mathbf{W}^{\top}=\mathbf{W}$, $\mathbf{W} \mathbf{1}_n=\mathbf{1}_n$, and $\mathbf{1}_n^{\top} \mathbf{W}=\mathbf{1}_n^{\top}$.
\end{assumption}

We make the following standard assumptions on the loss functions. These assumptions are commonly adopted in the literature on decentralized online/offline optimization\cite{cao2023decentralized,tu2022distributed,liu2025decentralized} and federated learning\cite{xin2020variance} .

\begin{assumption}
For any $t \ge 0$, the global loss function $f_t$ is $\mu$-strongly convex and the local loss functions 
$f_{i, t}$ are $L_i$-smooth, i.e., for any $u,v \in\mathbb{R}^m$, 
\begin{align}
&\left\langle\nabla f_t(u)-\nabla f_t\left(v\right), u-v\right\rangle \geq \mu\left\|u-v\right\|^2,\\
&\left\|\nabla f_{i, t}(u)-\nabla f_{i, t}(v)\right\| \leq L_i \|u-v \|.
\end{align}
\label{smooth}
\end{assumption}
Under \cref{smooth}, the global loss function $f_t = \frac{1}{n}\sum_{i=1}^n f_{i,t}$ is $L$-smooth, where $L=\max_{i \in \mathcal{V}} \{L_i\}$. Let $\bar{g}_t := \frac{1}{n}\sum_{i=1}^n \nabla f_{i,t}(x_{i,t})$ be the average of local gradients at local points $X_t$. The deviation between this average local gradient and the global gradient at the average point $\bar{x}_t$ can be bounded: $\mathbb{E}[\|\bar{g}_t - \nabla f_t(\bar{x}_t)\|] \le \frac{L}{\sqrt{n}}\|X_t-\mathbf{1}\bar{x}_t^\top\|$. The $\mu$-strong convexity ensures that the global problem $\min_x f_t(x)$ has a unique minimizer, denoted by $x_t^\star \in \mathbb{R}^{m}$, for each $t$.

\begin{assumption}[Bounded Gradient Heterogeneity at Optimum]
There exists a sequence of non-negative constants $\{\sigma_{i,t}^2\}_{t \ge 0}$ such that for any agent $i \in \mathcal{V}$ and time steps $t \ge 0$,
$$
\|\nabla f_{i,t}(x_t^\star)\|^2 \le \sigma_{i,t}^2.
$$
The average heterogeneity at time $t$ is defined as $\bar{\sigma}_t^2:=\frac{1}{n}\sum_{i=1}^n \sigma_{i,t}^2$.
\label{assump:heterogeneity}
\end{assumption}

Building upon Assumption \ref{assump:heterogeneity}, which generalizes the concept of bounded gradient heterogeneity from \cite{koloskova2020unified} to accommodate time-varying local objective functions, this work establishes a foundational framework for analyzing distributed optimization in dynamic environments.  This assumption characterizes the dissimilarity among local objective functions at the global optimum by bounding the gradient norm of each agent’s objective function $f_{i,t}$ at the optimal point $x_t^{\star}$. The heterogeneity is captured through individual constants $\sigma_{i,t}^2$ and summarized by the aggregated measure $\bar{\sigma}_t^2$. Notably, $\bar{\sigma}_t^2 = 0$ when all local functions are identical at time $t$, signifying complete consensus among agents. Although similar bounded diversity assumptions have been employed in distributed optimization studies (e.g., \cite{lian2017can, tang2018d}), the formulation adopted here is distinct in that it constrains only the gradient behavior at the optimal point $x_t^{\star}$, rather than across the entire domain. This results in a less restrictive yet analytically meaningful foundation for addressing time-varying distributed optimization problems.

For distributed online optimization, dynamic regret is used to evaluate the performance of online algorithms, introduced as follows.

\begin{definition}[Dynamic Regret] \label{def:dynamic_regret}
Given the sequence of local decisions $\{x_{i, t}\}$ generated by an online distributed algorithm, the dynamic regret over $T$ time steps is defined as
\[
R_T \triangleq \sum_{t=1}^T \mathbb{E}[f_t\left(\bar{x}_t\right)] - \sum_{t=1}^T f_t\left(x^{\star}_t\right),
\]
where $\bar{x}_t \triangleq \frac{1}{n} \sum_{i=1}^n x_{i,t}$ denotes the average decision across all agents at time $t$, and $\{x_t^\star\}_{t \ge 1}$ represents the sequence of minimizers of the time-varying global objective functions $f_t(x)$. The expectation $\mathbb{E}[\cdot]$ is taken over the algorithm's internal randomness, such as that introduced by bandit feedback or stochastic gradient estimators.
\end{definition}


The static regret measures cumulative the loss compared to the best fixed decision retrospectively, i.e., $\sum_{t=1}^T f_t(x_t) - \min_{x \in \mathcal{X}} \sum_{t=1}^T f_t(x)$,  a metric that makes sense in stationary or slowly changing environments.   However, there are many real-world applications, such as distributed sensing, resource allocation for communication networks, and online advertising, where the data distributions and objectives are inherently non-stationary, which results in optimal solution $x_t^\star$ varying significantly over time. In these dynamic environments, an algorithm that converges to a fixed point may have low static regret but still be incapable of effectively adapting to the environment. The dynamic regret serves as a more robust and practical performance measure that benchmarks against the sequence of optimal solutions $\{x_t^\star\}_{t=1}^T$.

The dynamic regret of the proposed algorithm is characterized by two key regularity measures: $\mathbf{p}_t$, which quantifies the maximum gradient variation across functions and points, and $\mathbf{v}_t$, which captures the distance between consecutive optimal points in the time-varying optimization scenario, defined as follows
\begin{align}
& \mathbf{p}_t \triangleq \sup _{i\in \mathcal{V}} \sup _{x \in \mathbb{R}^m}\left\|\nabla f_{i,t+1}(x)-\nabla f_{i,t}(x)\right\|, \\
& \mathbf{v}_t\triangleq\left\|x^{\star}_{t+1}-x^{\star}_t\right\| .
\end{align}

\subsection{Compression Methods}

We consider compression operators $\mathcal{C}:\mathbb{R}^m\rightarrow \mathbb{R}^m$ that map a vector to a potentially lower-information representation suitable for communication. These operators can be unbiased (e.g., certain sparsifiers) or biased (e.g., quantization, Top-k) \cite{koloskova2019decentralized,xin2020variance}. We assume the following property regarding the compression error.

\begin{assumption}\label{compress}
There exists a constant $\omega \in[0,1]$ such that for any vector $x \in \mathbb{R}^m$, the compression operator $\mathcal{C}$ satisfies
$$
\mathbb{E}_{\mathcal{C}}\left[\|\mathcal{C}(x)-x\|^2\right] \leq \omega \|x\|^2.
$$
The expectation $\mathbb{E}_{\mathcal{C}}[\cdot]$ is taken with respect to the internal randomness of the operator $\mathcal{C}$ (if any). The value $\omega=0$ corresponds to no compression, while $\omega=1$ represents the loosest bound.
\end{assumption}

Examples of operators satisfying \cref{compress} include:

\textbf{Stochastic quantizer.} For a given vector $\boldsymbol{x} \in \mathbb{R}^m$ and a positive integer $s$ (number of quantization levels per sign), a common stochastic quantizer $\mathcal{C}_Q$ operates entry-wisely. For each entry $x_i$, it computes $l_i = \lfloor s |x_i| / \|\boldsymbol{x}\| \rfloor$. Then,
\begin{align}
[\mathcal{C}_Q(\boldsymbol{x})]_i = 
\begin{cases}
\displaystyle \frac{\|\boldsymbol{x}\|}{s} \operatorname{sgn}(x_i) \cdot l_i & \text{w.p. } 1 - \delta_i, \\
\displaystyle \frac{\|\boldsymbol{x}\|}{s} \operatorname{sgn}(x_i) \cdot (l_i + 1) & \text{w.p. } \delta_i,
\end{cases}
\end{align}
where $\delta_i = \frac{s |x_i|}{\|\boldsymbol{x}\|} - l_i$. This quantizer is unbiased, i.e., $\mathbb{E}[\mathcal{C}_Q(\boldsymbol{x})] = \boldsymbol{x}$, and satisfies \cref{compress} with $\omega = \min\{m/s^2, \sqrt{m}/s\}$.

\textbf{Top-$k$ sparsification.} The Top-$k$ operator, $\mathcal{C}_{\text{top}}(\mathbf{x})$, retains the $k$ elements of $\mathbf{x}$ with the largest absolute values and sets the rest to zero. This can be expressed as $\mathcal{C}_{\text{top}}(\mathbf{x}) = \mathbf{x} \odot \mathbf{p}$, where $\mathbf{p}$ is a binary mask with $k$ ones corresponding to the largest magnitude entries. Top-$k$ is a biased operator (unless $\mathbf{x}$ is already $k$-sparse) and satisfies \cref{compress} with $\omega = 1 - k/m$.

\subsection{Random Gradient-Free Optimization} \label{RGF}

When gradients $\nabla f_{i,t}(x)$ are unavailable, we resort to gradient-free methods that estimate gradients using only function value evaluations. The one-point bandit method employs Gaussian smoothing and randomized finite differences. The gradient of the smoothed function $f_{i,t}^v(x) = \mathbb{E}_{u \sim \mathcal{N}(0, I_m)}[f_{i,t}(x + v u)]$ is estimated using a single function evaluation perturbed along a random direction. Specifically, the gradient estimator for $f_{i,t}$ at point $x$ is
\begin{align}
   G_{i, t}(x,u_{i,t}) =\frac{f_{i, t}\left(x+v_t u_{i, t}\right)-f_{i, t}\left(x\right)}{v_t} u_{i, t}, \label{bandit feed}
\end{align}
where $u_{i,t}$ is drawn from a standard multivariate Gaussian distribution $\mathcal{N}(0, I_m)$, and $v_t > 0$ is a smoothing parameter. The expectation of this estimator is the gradient of the smoothed function: $\mathbb{E}_{u_{i,t}}[G_{i, t}(x,u_{i,t})] = \nabla f_{i,t}^{v_t}(x)$. The following Lemma bounds the bias and variance of this estimator with respect to the true gradient $\nabla f_{i,t}(x)$.

\begin{lemma}\cite[Lemma 3, Theorem 4]{nesterov2017random}\label{bandit}
 If $f_{i, t}(x)$ is $L_i$-smooth, then for any $x \in \mathbb{R}^m$, smoothing parameter $v_t > 0$, and random direction $u_{i,t} \sim \mathcal{N}(0, I_m)$, the one-point gradient estimator $G_{i, t}(x, u_{i,t})$ defined in \cref{bandit feed} satisfies
\begin{align}
    &\mathbb{E}_{u_{i,t}}\left[\left\|G_{i, t}(x, u_{i,t})\right\|^2\right]\leq 2(m+4)\left\|\nabla f_{i, t}(x)\right\|^2  \nonumber\\ 
&\qquad\qquad\qquad\qquad\quad\quad\quad+\frac{v_t^2}{2} L_i^2(m+6)^3,\label{bandit_2}\\
      &\left\|\mathbb{E}_{u_{i,t}} \left[G_{i, t}(x, u_{i,t})\right]-\nabla f_{i, t}(x)\right\| \leq \frac{v_t}{2} L_i (m+3)^{\frac{3}{2}}.  \label{bandit_1}
\end{align}

\end{lemma}

\section{Online Compressed Gradient Tracking}\label{ocgt}

In this section, we propose an algorithm framework suitable for communication-constrained settings,  supporting zero-order/stochastic gradient and maintaining effective consensus among agents. We also present a dynamic regret analysis that establishes performance guarantees under standard assumptions.

\subsection{Algorithm Description}
\begin{algorithm}[t]
\caption{Compressed Communication Subroutine}
\label{alg:c-comm}
\begin{algorithmic}[1]
\State \textbf{Initialization:}  Local estimate $\theta_i \in \mathbb{R}^m$; Historical estimates $\hat{h}_i, \tilde{h}_i \in \mathbb{R}^m$; Mixing parameter $\alpha \in (0,1)$;
\State Compress  and communicate : $q_i = \mathcal{C}(\theta_i - \hat{h}_i)$

\State Update consensus estimate: $\tilde{\theta}_i = \tilde{h}_i + \sum_{j \in \mathcal{N}_i} w_{ij}q_j$

\State Update local estimate: $\hat{\theta}_i = \hat{h}_i + q_i$

\State Update historical estimates:
\begin{align}
	\tilde{h}_i^+ = \alpha \tilde{\theta}_i + (1-\alpha)\tilde{h}_i \nonumber\\
\hat{h}_i^+ = \alpha \hat{\theta}_i + (1-\alpha)\hat{h}_i \nonumber
\end{align}
\State \textbf{Return:} $\hat{\theta}_i, \tilde{\theta}_i, \hat{h}_i^+, \tilde{h}_i^+$
\end{algorithmic}
\end{algorithm}

\begin{algorithm}[ht]
\caption{\textbf{Unified OCGT Framework: \\OCGT-BF/OCSGT (Option I/II)}}
\label{alg:unified-ocgt}
\begin{algorithmic}[1]
\State \textbf{Initialization:}  Smoothing parameter $v_t > 0$; stepsize $\eta \in (0,1)$; mixing parameters $\alpha_x, \alpha_y \in (0,1)$; initial point $x_{i,0} \in \mathbb{R}^m$, $\hat{h}_{i,0}^x, \hat{h}_{i,0}^y \in \mathbb{R}^m$, $\tilde{h}_{i,0}^x = \sum_{j \in \mathcal{N}_i} w_{ij}\hat{h}_{j,0}^x$, 
and $\tilde{h}_{i,0}^y = \sum_{j \in \mathcal{N}_i} w_{ij}\hat{h}_{j,0}^y$ for each $i \in \mathcal{V}$.

\For{each agent $i \in \mathcal{V}$ in parallel}
    \For{$t=0,1,\ldots,T-1$}
        \State \textbf{(1) Compression and communication step:}
        \begin{align}
            &\left(\hat{x}_{i,t}, \tilde{x}_{i,t}, \hat{h}_{i,t+1}^x, \tilde{h}_{i,t+1}^x\right)\nonumber \\ &\quad\quad\quad\quad= \text{Algorithm \ref{alg:c-comm} } \left(x_{i,t}, \hat{h}_{i,t}^x, \tilde{h}_{i,t}^x, \alpha_x \right)\nonumber\\
            &\left(\hat{y}_{i,t}, \tilde{y}_{i,t}, \hat{h}_{i,t+1}^y, \tilde{h}_{i,t+1}^y\right) \nonumber \\
            &\quad\quad\quad\quad= \text{Algorithm \ref{alg:c-comm} }\left(y_{i,t}, \hat{h}_{i,t}^y, \tilde{h}_{i,t}^y, \alpha_y\right) \nonumber
        \end{align}
        \vspace{-1em}
        \State \textbf{(2) Update decision variable:}
        \[
        x_{i,t+1} 
        \;=\;
        x_{i,t} \;-\; \eta\,\Bigl(\,\hat{x}_{i,t} \;-\; \tilde{x}_{i,t} \;+\; y_{i,t}\Bigr).
        \]

        \State \textbf{(3) Gradient computation:}
        \begin{mdframed}[backgroundcolor=blue!15]
        \noindent
        \textbf{Option I (OCGT-BF: Bandit Feedback)}
        \begin{itemize}
            \item Generate a random direction: 
            \[
            u_{i,t+1} \sim \mathcal{N}(0, I_d).
            \]
            \item Compute gradient approximation by (\ref{bandit feed}) (one-point bandit):
            \[
            G_{i,t+1} \;=\; G_{i,t+1}\bigl(x_{i,t+1},\,u_{i,t+1}\bigr).
            \]
        \end{itemize}
        \end{mdframed}
        \begin{mdframed}[backgroundcolor=green!15]
        \noindent
        \textbf{Option II (OCSGT: Stochastic Gradient)}
        \begin{itemize}
            \item Compute stochastic gradient:
            \[
            G_{i,t+1} 
            \;=\; \nabla f_{i,t+1}\!\Bigl(x_{i,t+1},\, \xi_{i,t+1}\Bigr).
            \]
        \end{itemize}
        \end{mdframed}

        \State \textbf{(4) Update gradient tracking variable:}
        \[
        y_{i,t+1} 
        \;=\; 
        y_{i,t} \;-\; \eta\,\bigl(\hat{y}_{i,t} - \tilde{y}_{i,t}\bigr)
        \;+\; 
        \bigl(G_{i,t+1} - G_{i,t}\bigr).
        \]
    \EndFor
\EndFor

\State \textbf{Return:} The iterates $\{x_{i,t}\}_{t=0}^T$ for each agent $i\in \mathcal{V}$.

\end{algorithmic}
\end{algorithm}

This subsection proposes a unified  Online Compressed Gradient Tracking (\textbf{OCGT}) framework for distributed online optimization in multi-agent systems.

The proposed framework comprises two primary components: the Compressed Communication Subroutine (as detailed in \cref{alg:c-comm}), which enables efficient compressed communication among networked agents, and the unified \textbf{OCGT} algorithm (as detailed in \cref{alg:unified-ocgt}), which integrates either bandit or stochastic gradient feedback with consensus-based gradient tracking. Below is a detailed and cohesive description of each component.

The main objective of the Compressed Communication Subroutine is to reduce the communication overhead while maintaining a consensus mechanism across agents.
It takes as inputs a local estimate $\theta_i \in \mathbb{R}^d$, historical estimates $\hat{h}_i, \tilde{h}_i \in \mathbb{R}^d$, and a mixing parameter $\alpha \in (0,1)$. In line 1,  the difference \(\theta_i - \hat{h}_i\) is compressed via a chosen compression operator \(\mathcal{C}(\cdot)\), which can be either biased or unbiased, on condition that  Assumption \ref{compress} is satisfied. The result is a compressed update $q_i$, which each agent sends to its neighbors. Next, in line 2, the local agent aggregates the compressed updates received from its neighbors, scaled by weights $w_{ij}$, to form a consensus-based estimate $\tilde{\theta}_i$. In line 3, the local estimate $\hat{\theta}_i$ is updated by incorporating the locally compressed value $q_i$. Finally,  both $\hat{h}_i$ and $\tilde{h}_i$ are refreshed through a momentum-like update using the parameter $\alpha$. This design is inspired by a class of compressed distributed optimization algorithms \cite{seide20141, richtarik2021ef21, liao2022compressed, zhao2024faster}. Recent works \cite{10833724,2025107857} analyze nonlinear perturbations such as log-scale quantization, where convergence is established under a rigid sector-bound property of the quantizer. In contrast, Algorithm 1 has broader applicability and is more robust to noise because it leverages historical information.

Building on the Compressed Communication Subroutine, the unified \textbf{OCGT} algorithm addresses two settings distinguished by gradient feedback type: bandit feedback (\textbf{OCGT-BF}) and stochastic gradient (\textbf{OCSGT}).
Each agent starts with initial decision variables $x_{i,0}$ and gradient tracking variables $\hat{h}_{i,0}^x, \hat{h}_{i,0}^y$, along with their consensus counterparts. At each iteration, agents first compress and communicate both the local decision variables $x_{i,t}$ and tracking variables $y_{i,t}$. The mixing parameters $\alpha_x$ and $\alpha_y$ may differ to allow more flexible tuning of the update dynamics for decision and gradient variables. After communication, agents update their local decision variables by combining two terms:  a correction proportional to \(\hat{x}_{i,t} - \tilde{x}_{i,t}\), which reflects the deviation between local and consensus estimates, and a gradient descent-like term with stepsize \(\eta\).  



The gradient computation distinguishes the two variants within the unified framework: In the \textbf{OCGT-BF} variant (Option I), bandit feedback is integrated via Gaussian smoothing. Each agent samples a random direction $u_{i,t}$ from a standard normal distribution and computes the gradient approximation $G_{i,t}$ through a finite difference approach. In the \textbf{OCSGT} variant (Option II), each agent directly computes stochastic gradients $G_{i,t}:= \nabla f_{i,t}(x_{i,t}, \xi_{i,t})$ using randomly sampled data ${\xi_{i,t}}$ available at iteration $t$.

Finally, the gradient tracking variable $y_{i,t}$ is updated to aggregate current and prior gradient approximations, refining the global gradient estimate despite compression effects.

\begin{remark}
    Recent studies \cite{10833724,2025107857} investigate nonlinear perturbations such as log-scale quantization through direct compression, expressed as \( q_i = C(\theta_i) \), where convergence is guaranteed under a strict sector-bound condition on the quantizer. In contrast, our work focuses on online learning. The proposed error-compensated framework utilizes the mechanism \( q_i = C(\theta_i - h_i) \), where \( h_i \) preserves historical information to enable a broader range of compressors by actively correcting quantization bias, rather than relying on the rigid structural constraints imposed by the nonlinearity. Consequently, the approach mitigates error accumulation in dynamic environments and achieves regret bounds of the same order as uncompressed methods, while offering greater flexibility in compressor selection. In contrast to \cite{10833724}, which addresses directed balanced graphs, our analysis is currently limited to undirected graphs, and extending it to directed settings remains an important direction for future work.
\end{remark}

\begin{remark}
    At each iteration, each agent performs only first-order operations (vector additions, scalar–vector multiplications, and one local oracle call) together with a compressed consensus step with its neighbors, leading to a per-iteration cost of order \(\mathcal{O}(m(1+|\mathcal{N}_i|))\) for an \(m\)-dimensional decision variable and node degree \(|\mathcal{N}_i|\). Aggregated over a network with \(n\) agents, \(E\) edges, and a horizon of \(T\) iterations, the total arithmetic cost is thus \(\mathcal{O}(T m (n+2E))\) along with the cost of local oracle evaluations, which is polynomial in the network scales. Therefore, the resulting computational complexity is suitable for real-time implementation in large-scale systems.
\end{remark}

\subsection{Dynamic Regret Analysis}
We now analyze the performance of \cref{alg:unified-ocgt}, focusing on the variance of the gradient approximation and the resulting dynamic regret bounds.

The following lemma bounds the average variance of the one-point bandit gradient estimator used in \textbf{OCGT-BF}, relating it to consensus error, optimality gap, and gradient heterogeneity.

\begin{lemma}\label{Estimated variance}
Let Assumptions \ref{smooth} and \ref{assump:heterogeneity} hold. Consider the gradient approximation \(G_{i,t}(x_{i,t},u_{i,t})\) generated by \textbf{OCGT-BF} for any \(i\in \mathcal{V}\). Then, for any \(t \ge 0\),
\begin{align*}
    &\frac{1}{n}\sum_{i=1}^n \mathbb{E}_{u_{i,t}}[\|G_{i,t}-\nabla f_{i,t}(x_{i,t})\|^2] \nonumber \\
    &\le v^2_t M_L + \frac{8(m+4) L^2}{n} \|X_t - \mathbf{1}\bar{x}_t\|^2 \nonumber \\
    &\quad + 8(m+4) L^2 \|\bar{x}_t - x_t^\star\|^2 + 4(m+4)\bar{\sigma}_t^2,
\end{align*}
where $M_L := L^2 \left(\frac{1}{2}(m+6)^3+\frac{1}{4}(m+3)^2\right)$. 
\end{lemma}

\begin{proof}
Let $G_{i,t} := G_{i,t}(x_{i,t},u_{i,t})$ be the estimated gradient and $\nabla f_{i,t}(x_{i,t})$ the local full gradient. The variance for agent $i$ is decomposed by
\begin{align}
    &\mathbb{E}_{u_{i,t}}[\|G_{i,t}-\nabla f_{i,t}(x_{i,t})\|^2] \nonumber \\
    &\quad = \mathbb{E}_{u_{i,t}}[\|G_{i,t}\|^2] - \|\mathbb{E}_{u_{i,t}}[G_{i,t}]\|^2 \nonumber \\
    &\qquad + \|\mathbb{E}_{u_{i,t}}[G_{i,t}] - \nabla f_{i,t}(x_{i,t})\|^2. \label{Le2_1}
\end{align}
From \cref{bandit} and letting $L=\max_i L_i$, the bias term is bounded by \cref{bandit_1}, i.e., 
\begin{align}
    \|\mathbb{E}_{u_{i,t}}[G_{i,t}] - \nabla f_{i,t}(x_{i,t})\|^2 \le \frac{v_t^2}{4} L^2 (m+3)^3. \label{Le2_bias}
\end{align}
Similarly, from \cref{bandit}, the variance term is bounded by 
\begin{align}
    &\mathbb{E}_{u_{i,t}}[\|G_{i,t}\|^2] - \|\mathbb{E}_{u_{i,t}}[G_{i,t}]\|^2 \nonumber \\
    &\quad \le \mathbb{E}_{u_{i,t}}[\|G_{i,t}\|^2] \nonumber \\
    &\quad \le 2(m+4)\left\|\nabla f_{i, t}(x_{i,t})\right\|^2 +\frac{v_t^2}{2} L^2(m+6)^3. \label{Le2_var_part}
\end{align}
Plugging Eqs. (\ref{Le2_bias}) and (\ref{Le2_var_part}) into \cref{Le2_1} yields
\begin{align}
    &\mathbb{E}_{u_{i,t}}[\|G_{i,t}-\nabla f_{i,t}(x_{i,t})\|^2] \nonumber \\
    &\quad \le v_t^2 L^2 \left(\frac{1}{2}(m+6)^3+\frac{1}{4}(m+3)^3\right) \nonumber \\
    & \qquad + 2(m+4)\|\nabla f_{i, t}(x_{i,t})\|^2. \label{Le2_single_agent}
\end{align}
Define $M_L := L^2 \left(\frac{1}{2}(m+6)^3+\frac{1}{4}(m+3)^3\right)$. We next bound the gradient term $\|\nabla f_{i, t}(x_{i,t})\|^2$. Using \cref{smooth} and the fact that $\|a+b\|^2 \le 2a^2+2b^2$, it holds that 
\begin{align}
    &\|\nabla f_{i, t}(x_{i,t})\|^2 \nonumber \\
    &\quad\le 2\|\nabla f_{i,t}(x_{i,t}) - \nabla f_{i,t}(x_t^\star)\|^2 + 2\|\nabla f_{i,t}(x_t^\star)\|^2 \nonumber \\
    &\quad\le 2 L^2 \|x_{i,t} - x_t^\star\|^2 + 2\|\nabla f_{i,t}(x_t^\star)\|^2 \nonumber \\
    &\quad\le 4 L^2 \|x_{i,t} - \bar{x}_t\|^2 + 4 L^2 \|\bar{x}_t - x_t^\star\|^2 + 2\|\nabla f_{i,t}(x_t^\star)\|^2. \label{Le2_grad_bound}
\end{align}
Plugging \cref{Le2_grad_bound} into \cref{Le2_single_agent} gives
\begin{align}
    &\mathbb{E}_{u_{i,t}}[\|G_{i,t}-\nabla f_{i,t}(x_{i,t})\|^2] \nonumber \\
    &\le v^2_t M_L + 8(m+4) L^2 \|x_{i,t} - \bar{x}_t\|^2 \nonumber \\
    &\quad+ 8(m+4) L^2 \|\bar{x}_t - x_t^\star\|^2 + 4(m+4)\|\nabla f_{i,t}(x_t^\star)\|^2. \nonumber
\end{align}
Summing the above equation over $i$ and dividing by $n$ under \cref{assump:heterogeneity}, it can be derived that 
\begin{align}
    &\frac{1}{n}\sum_{i=1}^n \mathbb{E}_{u_{i,t}}[\|G_{i,t}-\nabla f_{i,t}(x_{i,t})\|^2] \nonumber \\
    &\le v^2_t M_L + \frac{8(m+4) L^2}{n} \|X_t - \mathbf{1}\bar{x}_t\|^2 \nonumber \\
    &\quad + 8(m+4) L^2 \|\bar{x}_t - x_t^\star\|^2 + 4(m+4)\bar{\sigma}_t^2.
\end{align}
This completes the proof.
\end{proof}

Define $\mathrm{U}_t := \big[ \Phi_t^x, \Phi_t^y, \Psi_t, \Omega_t^x, \Omega_t^y \big]^\top$,  where $\Phi_t^x:=\mathbb{E}[\|X_t-\mathbf{1} \bar{y}_t\|^2] $, $ \Phi_t^y:=\mathbb{E}[\|Y_t-\mathbf{1} \bar{y}_t\|^2]$, $\Psi_t:=\mathbb{E}[\|\bar{x}_t-x_i^\star\|^2] $, $\Omega_t^x:=\mathbb{E}[\|X_t-H_t^x\|^2]$ and $\Omega_t^y:=\mathbb{E}[\|Y_t-H_t^y\|^2] $. The following theorem demonstrates the dynamic regret of \textbf{OCGT-BF} algorithm.

\begin{lemma}\label{lem:system_ineq}
Let Assumptions \ref{graph}-\ref{compress} hold. Consider the sequence $\{x_{i,t}\}_{t=1}^T$ generated by \textbf{OCGT-BF}. For any agent $i \in \mathcal{V}$ and $t < T$, the following linear system of inequalities holds
\begin{align}
\mathrm{U}_{t+1} \preceq G(\eta)\mathrm{U}_t + \mathrm{V}_t + v_t^2 M_L \mathbf{d}_1 + \bar{\sigma}_t^2\mathbf{d}_2, \label{Lemm14_0}
\end{align}
where the symbol $\preceq$ represents an element-wise comparison between vectors. Moreover, there exists a constant $\eta^\star > 0$, defined in the supplementary file, such that for all  $\eta \in (0, \eta^\star]$, the spectral radius satisfies $0 < \rho(G(\eta)) < 1$. The vectors $\mathrm{V}_t, \mathbf{d}_1, \mathbf{d}_2 \in \mathbb{R}^{5}$ are defined as
    \begin{align}
        \mathrm{V}_t & := [ 0 , a_1\mathbf{p}_t^2+a_2\mathbf{v}_t^2 , a_4\mathbf{v}_t , 0 , a_yn \mathbf{p}_t^2+a_6\mathbf{v}_t^2 ]^\top, \nonumber\\
        \mathbf{d}_1 &:= \begin{bmatrix} 0 \\ a_3 \\ a_5 \\ a_x^\prime \\ a_7 \end{bmatrix}, \quad 
        \mathbf{d}_2 := 4(m+4) \begin{bmatrix} 0 \\ a_3 \\ a_5 \\ a_x^\prime \\ a_7 \end{bmatrix}, 
    \end{align}
    where the coefficients are defined in Appendix.
    
\end{lemma}
\begin{proof}
    See Appendix.
\end{proof}
    
\begin{theorem}\label{Th1}
Suppose Assumptions \ref{graph}-\ref{compress} hold. Consider the sequence $\{x_{i,t}\}_{t=1}^T$ generated by \textbf{OCGT-BF} for any agent $i\in \mathcal{V}$.  Then, for all  $\eta \in (0, \eta^\star]$ and $T>2$, such that 
$$
\begin{aligned}
&R_T \le \frac{L \lambda}{2(1-\tilde{\rho})} \Bigg( \|\mathrm{U}_0\| + C_p \sum_{t=0}^{T-1} \mathbf{p}_t^2 + C_v \sum_{t=0}^{T-1}\mathbf{v}_t^2 \\
        &\qquad + C_d M_L\sum_{t=0}^{T-1} v_t^2   + 4(m+4) C_d \sum_{t=0}^{T-1} \bar{\sigma}_t^2 \Bigg), 
\end{aligned}
$$
where 
$C_p, C_v$ and $C_d$ are constants defined as
$$ C_p = \sqrt{2a_1^2 + 2(a_y n)^2}, \quad
    C_v = \sqrt{2a_2^2 + a_4^2 + 2a_6^2},$$ and $C_d = \sqrt{a_3^2 + a_5^2 + (a_x^\prime)^2 + a_7^2}$.   
\end{theorem} 
\begin{proof}
See Appendix.
\end{proof}

\begin{remark} \label{rem:Th1_implications}
Theorem \ref{Th1} characterizes the dynamic regret of \textbf{OCGT-BF}. The bound consists of several terms: $\|\mathrm{U}_0\|/(1-\tilde{\rho})$ represents the influence of the initial errors, scaled by a factor related to the convergence rate $\tilde{\rho}$; $\sum \mathbf{p}_t^2$  captures the cumulative functional change over time; $\sum \mathbf{v}_t^2$ captures the cumulative drift of the optimal solution; $\sum \bar{\sigma}_t^2$ reflects the impact of gradient heterogeneity across agents at the optima; $\sum v_t^2$ arises from the bias and variance of the one-point bandit feedback, scaling with the smoothing parameter $v_t$ and dimension $m$ (via $M_L$). Assuming $v_t=v$ leads to a term $C_d M_L v^2 T$. Using a decaying $v_t$ (e.g., $v_t \propto 1/\sqrt{t}$ or $1/t$) could potentially improve this term's dependence on $T$.

If the environment is relatively stable (i.e., $\sum \mathbf{p}_t$, $\sum \mathbf{v}_t$, $\sum \bar{\sigma}_t^2$ grow sub-linearly with $T$) and a constant $v$ is used as in Theorem \ref{Th1}, the regret is dominated by $\mathcal{O}(v^2 T)$. The factor $1/(1-\tilde{\rho})$ relates to the geometric convergence of the underlying dynamics in the absence of noise and environmental changes, but the persistent noise/bias from bandit feedback and environmental shifts lead to the cumulative terms. The dimension dependence $m$ appears primarily through $M_L$ in the bandit noise term.
\end{remark}

\begin{remark}
    Theoretically, setting $v_t=\frac{1}{t}$ would decrease faster and eliminate the $T$ dependency. From the one-point gradient estimation (\ref{bandit feed}), smaller $v_t$ values yield more accurate gradient approximations, potentially reducing the distance from the local gradient significantly in \cref{bandit_1}. However, excessively small $v_t$ values introduce numerical stability challenges in practical implementations, necessitating a balanced approach that considers both theoretical optimality and computational feasibility.
\end{remark}

\begin{remark}
   Theorem 1 shows that the dynamic regret prefactor scales as \(1/(1-\tilde{\rho})\), where \(\tilde{\rho} = \rho(G(\eta))\) is the spectral radius that determines the geometric contraction rate of the recursion. The network topology affects the analysis only through the mixing matrix \(W\). The value of \(\tilde{\rho}\) increases with the magnitude of the second-largest eigenvalue of \(W\), which amplifies the cumulative consensus error. Empirical studies suggest that higher network clustering reduces algebraic connectivity. This reduction weakens the consensus mixing process and slows convergence \cite{Doostmohammadian2024Clustering}. As a result, graphs with stronger clustering tend to yield a larger \(1/(1-\tilde{\rho})\), leading to a tighter regret bound. In practice, the tuning of graph-dependent parameters such as \(\alpha_x\), \(\alpha_y\), and \(\eta^*\) should be guided by the network’s spectral properties.
\end{remark}


Define $\mathcal{F}_t$ as the $\sigma$-algebra that is generated by the history up to time $t-1$. The following assumption is commonly used in distributed stochastic optimization and federated learning \cite{xin2020variance}.

\begin{assumption}
    For any $t \ge 0$, the stochastic gradient of each $f_{i,t}$ is unbiased and has bounded variance, i.e.,
    \begin{align}
        &\mathbb{E}[\nabla f_{i,t}(x_{i,t},\xi_{i,t})\mid \mathcal{F}_t]=\nabla f_{i,t}(x_{i,t}), \nonumber\\
        &\mathbb{E}[\|\nabla f_{i,t}(x_{i,t},\xi_{i,t})-\nabla f_{i,t}(x_{i,t})\|^2 \mid \mathcal{F}_t] \le \hat{\sigma}^2_t, \nonumber
    \end{align}
where $\nabla f_{i,t}(x_{i,t},\xi_{i,t})$ is a stochastic gradient of the loss function  $f_{i,t}(x_{i,t})$ associated  with data sample $\xi_{i,t}$.
    \label{unbia}
\end{assumption}

We also establish the dynamic regret bound for \textbf{OCSGT}, as presented below.

\begin{theorem}\label{Th3} 
Suppose Assumptions \ref{graph}, \ref{smooth}, \ref{compress}, and \ref{unbia} hold.  Let $\{x_{i,t}\}_{t=1}^T$ be the sequence generated by \textbf{OCSGT} for any agent $i\in \mathcal{V}$. Then, for all  $\eta \in (0, \eta^\star]$ and $T>2$, such that 
$$
\begin{aligned}
R_T  \le &\frac{L \lambda}{2(1-\tilde{\rho})} \Big( \|\mathrm{U}_0\| + C_p^\prime \sum_{t=0}^{T-1} \mathbf{p}_t^2 + a_2^\prime \sum_{t=0}^{T-1} \mathbf{v}_t^2 \\
    &+ C_d^\prime \sum_{t=0}^{T-1} \hat{\sigma}_t^2 \Big), 
\end{aligned}
$$
where 
$C_p^\prime$,  $C_d^\prime$ and $a_2^\prime$ are constants that depend on parameters such as $L, \mu, m, n$, and $\rho_{\eta}$.
\end{theorem}
\begin{proof}
    See Appendix.
\end{proof}
\begin{remark}
One can compare Theorem \ref{Th3}  with Theorem \ref{Th1} (bandit feedback) in detail. 
The stochastic gradient version (Theorem \ref{Th3}) replaces the terms related to bandit estimation error ($C_d M_L \sum v_t^2$) and explicit gradient heterogeneity ($4(m+4) C_d \sum \bar{\sigma}_t^2$) from Theorem \ref{Th1} with a single term reflecting the variance of the stochastic gradients ($C_d^\prime \sum \hat{\sigma}_t^2$). Crucially, the bound in Theorem \ref{Th3} does not explicitly depend on the dimension $m$, different from the $M_L$ term in Theorem \ref{Th1}, which scales polynomially with $m$.  If full local gradients are available (i.e., $\hat{\sigma}_t^2 = 0$ for all $t$), the regret bound becomes $\mathcal{O}\left(\frac{\|\mathrm{U}_0\| + C_p^\prime \sum \mathbf{p}_t^2 + a_2^\prime \sum \mathbf{v}_t^2}{1-\tilde{\rho}}\right)$. This recovers a dynamic regret bound similar in structure to that of the original Online Gradient Tracking (\textbf{OGT}) algorithm \cite{zhang2019distributed}, which uses full gradients and no compression.
\end{remark}

\begin{table}[b]
\centering
\caption{Comparison of Algorithm Configurations}
\label{tab:algorithm_configs}
\begin{tabular}{lccc}
\toprule
\textbf{Algorithm} & \textbf{Compression} & \textbf{Feedback Type} & \textbf{Parameter \(v\)} \\
\midrule
\textbf{OGT} \cite{zhang2019distributed} 
  & No & Local Full Gradient & - \\
\textbf{DC-DOGD} \cite{tu2022distributed} 
  & Yes & Local Full Gradient & - \\
\textbf{Wang24} \cite{wang2024distributed} 
  & No & One-point Feedback & \(v = 0.9/t\) \\
\textbf{OCSGT}  
  & Yes & Stochastic Gradient & - \\
\textbf{OCGT-BF}   
  & Yes & One-point Feedback &  \(v = 0.1\) \\
\textbf{OCGT-BF}  
  & Yes & One-point Feedback &  \(v = 0.9/t\) \\
\bottomrule
\end{tabular}
\end{table}

\section{Experiments}\label{experiments}

In this section, we present comprehensive experimental results to validate the performance of the proposed decentralized optimization algorithms under a logistic regression classification task on the spam dataset \cite{metsis2021enron}. Consider a loss function of the form \(f\bigl(\theta,\xi^i\bigr)=\tfrac{1}{M^i}\sum_{s=1}^{M^i}\bigl((1-b_s^i)(a_s^i)^\top\theta - \log\bigl(s\bigl((a_s^i)^\top\theta\bigr)\bigr)\bigr) + \tfrac{r^i}{2}\|\theta\|^2\), where \(M^i\) is the number of samples for agent 
 \(i\), \(r^i\) is a positive regularization parameter inversely proportional to \(M^i\), and \(s(a)=\tfrac{1}{1+e^{-a}}\) denotes the sigmoid function. At each iteration, we randomly select 30 samples from the spam dataset and distribute them in a decentralized manner to 30 agents, ensuring that each agent processes one unique sample per round. The underlying communication topology is an Erdos–Rényi~(E–R) graph \cite{erdds1959random} with connection probability \(p=0.15\).

\begin{figure}[t]
  \centering
  \begin{minipage}{0.85\columnwidth}
    \includegraphics[width=\linewidth]{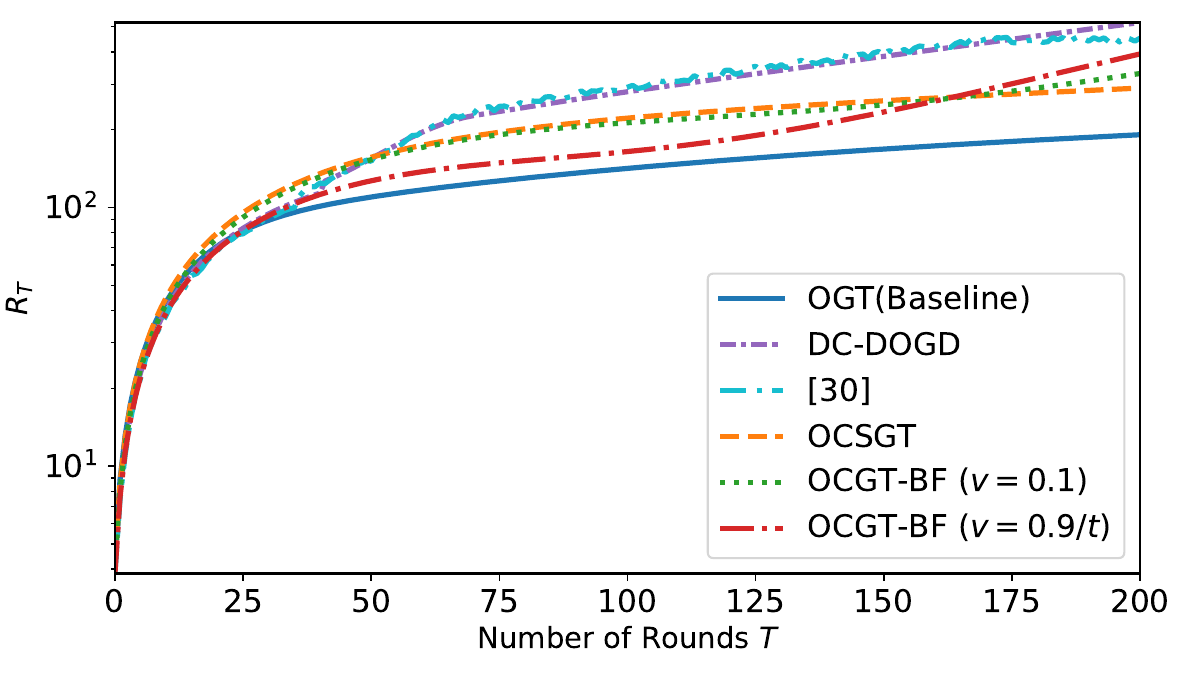}
    \caption*{(a)}
  \end{minipage}
  \hfill
  \begin{minipage}{0.85\columnwidth}
    \includegraphics[width=\linewidth]{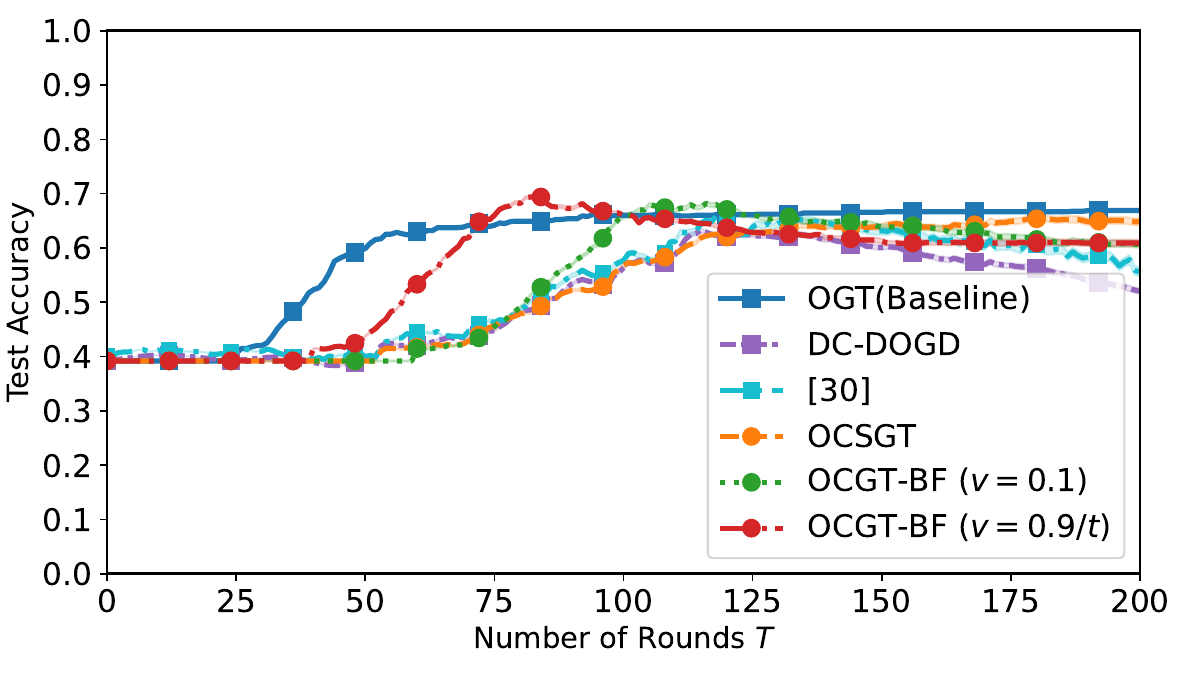}
    \caption*{(b)}
  \end{minipage}
  
\caption{Dynamic regret \(R_T\) comparisons (a) and Test accuracy trajectories (b)  among \textbf{OGT}, \textbf{DC-DOGD}, \cite{wang2024distributed}, \textbf{OCSGT}, and two variants of \textbf{OCGT-BF}, using a unified stepsize \(\eta = 0.01\), mixing parameter \(\alpha = 0.5\), and smoothing parameter \(v = 0.1\).}
    \label{fig1}
\end{figure}

We compare five algorithms: \textbf{OGT} \cite{zhang2019distributed}, \textbf{DC-DOGD}\cite{tu2022distributed} , \cite{wang2024distributed}, \textbf{OCSGT}, and two variants of \textbf{OCGT-BF}, each employing a distinct smoothing parameter $v$ within the one-point feedback mechanism. The \textbf{OGT} and \textbf{DC-DOGD} algorithm utilize local full gradient information without any compression, while \textbf{OCSGT} incorporates stochastic gradient estimates alongside a communication compression step. In contrast, \cite{wang2024distributed} and \textbf{OCGT-BF} operate solely on one-point bandit feedback, where the smoothing parameter $v$ determines the level of Gaussian perturbation. One variant adopts a fixed setting $v = 0.1$, whereas the other employs a time-decaying scheme $v = 0.9/t$. A comprehensive summary of these configurations is presented in \cref{tab:algorithm_configs}. In all implementations, the stepsize and mixing parameters are appropriately tuned to ensure stable convergence over $200$ training epochs, with model synchronization performed at each communication round according to the respective algorithmic update rules.


As shown in Fig. \ref{fig1}(a), which depicts the regret $R_T$ against the number of communication rounds \(T\), \textbf{OGT} achieves the lowest overall regret, reflecting the benefit of having direct access to uncompressed gradients. The compressed method \textbf{OCSGT} exhibits slightly higher regret, consistent with the trade-off between communication savings and gradient accuracy; \textbf{DC-DOGD} shows a regret trend similar to \textbf{OCSGT}. Both \textbf{OCGT-BF} variants also accumulate higher regret compared with \textbf{OGT}, which stems from the one-point stochastic approximation introducing additional variance. \cite{wang2024distributed}, which also relies on one-point feedback but without compression, has a regret trajectory comparable to the \textbf{OCGT-BF} variants. In addition, Fig. \ref{fig1}(b) illustrates that all methods achieve similar test accuracy patterns, with \textbf{OGT} converging marginally faster yet eventually residing in a performance range comparable to the other methods. The test accuracies of \textbf{DC-DOGD} and \cite{wang2024distributed} are also comparable to the other algorithms, indicating their effectiveness in this classification task. It is worth noting that \textbf{OCGT-BF} with the decaying smoothing parameter (i.e., $\nu=0.9/t$) demonstrates competitive accuracy while reducing information exchange through bandit-based updates.

  \begin{figure}[t]
    \centering
    \includegraphics[width=0.85\columnwidth]{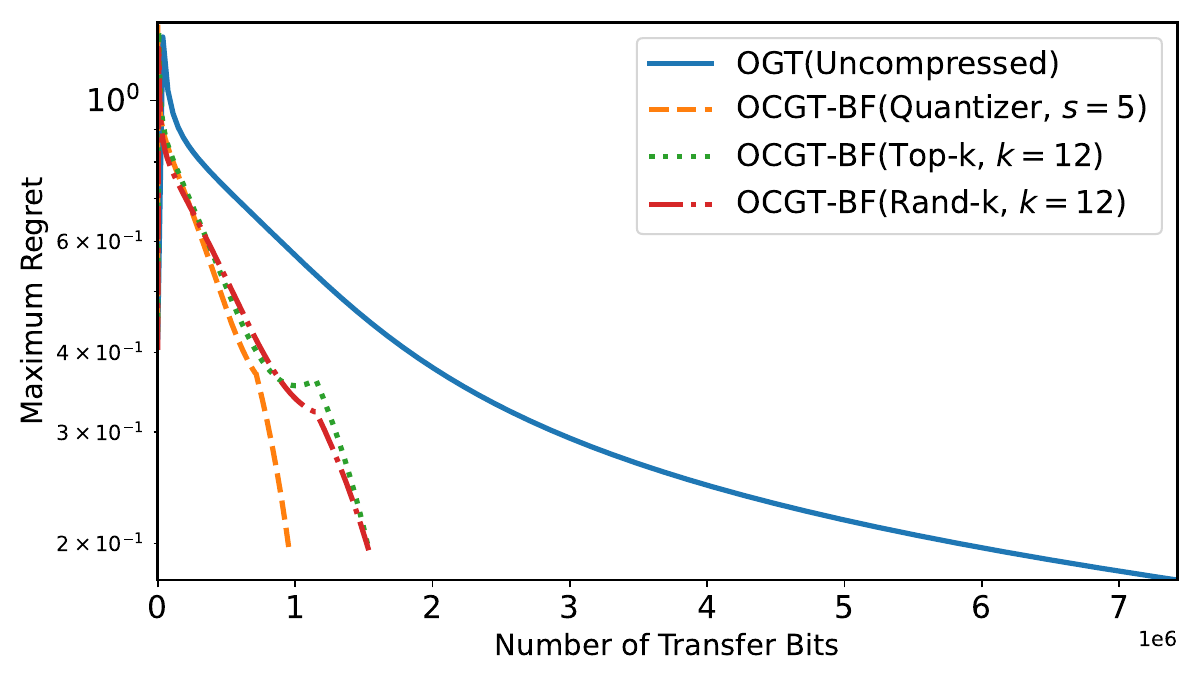}
    \caption{Comparison of the maximum regret $\max_{i \in \mathcal{V}} \mathcal{R}_i(T)/T$ versus the number of transfer bits for \textbf{OGT}~ (uncompressed) and three \textbf{OCGT-BF} variants (quantizer-based, top-k, and random-k).}
    \label{fig3}
\end{figure}

In a further experiment, we examine how each algorithm’s maximum regret (i.e., $\max _{i \in \mathcal{V}} \mathcal{R}_i(T) / T$ ) scales with respect to the total number of transmitted bits.  To this end, we augment our compressed methods with a stochastic quantizer, enabling us to finely control the bit budget per communication round.  As illustrated in Fig. \ref{fig3}, the uncompressed \textbf{OGT} requires significantly more bits to attain a similar level of regret, underscoring the communication savings afforded by compression.  Among the compressed approaches, \textbf{OCGT-BF} with a relatively coarse quantizer (e.g., $s=5$) demonstrates the steepest initial drop in maximum regret for small transfer-bit budgets, highlighting the effectiveness of stochastic quantization in high-compression regimes.  Meanwhile, both \textbf{OCGT-BF~(Top-k)} and \textbf{OCGT-BF~(Rand-k)} balance the trade-off between bit usage and regret reduction: by selecting and transmitting only a subset ($k=12$) of coordinate updates, they achieve near-competitive regret at a fraction of the communication cost.  Overall, these results reinforce that carefully designed quantization or sparsification schemes can markedly reduce bandwidth requirements while preserving favorable convergence behavior.

\section{Conclusions}\label{conclusions}
In this paper, we investigate distributed online convex optimization under communication constraints and limited gradient feedback, and propose an \textbf{OCGT} framework to address these practical challenges. By integrating a novel compression scheme with error compensation, \textbf{OCGT} substantially reduces communication overhead while maintaining robust convergence properties. Moreover, our analysis rigorously characterizes the dynamic regret of the unified framework under both bandit and stochastic gradient scenarios, demonstrating that the algorithms can effectively track evolving optimal solutions in challenging, online environments. Extensive experimental evaluations, including logistic regression simulations on real datasets, further validate the scalability and adaptability of the proposed algorithms in machine learning tasks.  Finally, it is important to note that the current work only focuses on doubly stochastic networks. Extending this work to more complex and robust network topologies and incorporating momentum-based methods for faster convergence is a fruitful future research direction.

\section*{Appendix}
\label{App-CA}
\CenteredUppercase{Some Basic Lemmas}

The following vector and matrix inequalities will be invoked in the later analysis.
\begin{lemma}\cite{liao2022compressed}
    For any integer $b \ge 1 $ and vector $\mathbf{u}_i \in \mathbb{R}^{p}$, we have
    \begin{equation}
        \|\sum_{i=1}^{b}\mathbf{u}_i\|^2 \le b\sum_{i=1}^{b}\|\mathbf{u}_i\|^2. \label{young-1}
    \end{equation} In particular, for any constant $ \zeta >1 $, one gets
    \begin{align}
        \|\sum_{i=1}^{b}\mathbf{u}_i\|^2 \leq \zeta\|\mathbf{u_1}\|^2 + \frac{(b-1) \zeta}{\zeta-1}\sum_{i=2}^{b}\|\mathbf{u}_i\|^2.\label{young-2}
    \end{align}
\label{young}
\end{lemma}

\begin{lemma}[\cite{qu2017harnessing}]\label{sc-lemma}
Let Assumption \ref{smooth} hold. If $0<\eta < \frac{2}{L}$, then for all $x \in \mathbb{R}^m$,
\begin{align}
    \|x-\eta \nabla f(x)-x^{*}\| \leq(1-\mu \eta)\|x-x^{*}\|,
\end{align}
where $\mu$ and $L$ are the strongly convex and Lipschitz smooth constants of the global loss function $f$, respectively.
\end{lemma}

\label{pfth1}
\CenteredUppercase{Some necessary auxiliary results}

We first provide supporting lemmas for the proof. Updates of the decision variables and gradient tracking variables  can be written in matrix form. Let $\mathbf{W}_\eta := \mathbf{I} - \eta(\mathbf{I}-\mathbf{W})$.
\begin{align}
X_{t+1} &= \mathbf{W}_\eta X_t - \eta Y_t +\eta (\mathbf{I}-\mathbf{W})(X_t-\widetilde{X}_t), \label{update_X}\\
 Y_{t+1} &=\mathbf{W}_\eta Y_t +G_{t+1}-G_t + \eta (\mathbf{I}-\mathbf{W})(Y_t-\widetilde{Y}_t). \label{update_Y}
\end{align}
The average iterates evolve as
\begin{align}
    &\bar{x}_{t+1}=\bar{x}_t-\eta\bar{y}_t, \label{update_bar_x}\\  
    &\bar{y}_t = \frac{1}{n}\mathbf{1}^\top Y_t = \frac{1}{n}\mathbf{1}^\top G_t. \label{update_bar_y} 
\end{align}
Let $Q_t^x := \mathcal{C}(X_t - H_t^x)$ and $Q_t^y := \mathcal{C}(Y_t - H_t^y)$. Then $\widetilde{X}_t = Q_t^x + H_t^x$ and $\widetilde{Y}_t = Q_t^y + H_t^y$. The history updates can be written as
\begin{align}
    H_{t+1}^x &= (1-\alpha_x) H_t^x +\alpha_x \widetilde{X}_t, \label{update_hx}\\
    H_{t+1}^y &= (1-\alpha_y) H_t^y +\alpha_y \widetilde{Y}_t.
\end{align}
From  Assumption \ref{compress}, the compression errors satisfy
\begin{align}
    \mathbb{E}[\|X_t-\widetilde{X}_t\|^2 ]&\le \omega \mathbb{E}[\|X_t - H_t^x\|^2], \label{error_x}\\
    \mathbb{E}[\|Y_t-\widetilde{Y}_t\|^2 ]&\le \omega \mathbb{E}[\|Y_t - H_t^y\|^2]. \label{error_y}
\end{align}

\begin{lemma}\label{XX}
    Let Assumptions \ref{graph}-\ref{compress} hold. Consider $\{x_{i,t}\}_{t=1}^T$ generated by \textbf{OCGT-BF} for any agent $i\in \mathcal{V}$,
    \begin{align}
        &\mathbb{E}[\|X_{t+1}-X_t\|^2] \nonumber \\
        &\leq \left(12+12L^2\eta^2(8m+33)\right)\mathbb{E}[\|X_t-\mathbf{1}\bar{x}_t\|^2] \nonumber\\
        &\quad + 48nL^2\eta^2(m+4)\mathbb{E}[\|\bar{x}_t-x_t^*\|^2] + 12n\eta^2 v_t^2 M_L \nonumber\\
        &\quad +12\eta^2\mathbb{E}[\|Y_t-\mathbf{1}\bar{y}_t\|^2] + 12\omega \eta^2 \mathbb{E}[\|X_t-H_t^x\|^2] \nonumber\\
        &\quad + 48 n \eta^2 (m+4) \bar{\sigma}_t^2.
    \end{align}
\end{lemma}

\begin{proof}
    From the update rule \cref{update_X}, one has
    \begin{align*}
        &\|X_{t+1}-X_t\|^2 \\
        &= \|(\mathbf{W}_\eta -\mathbf{I})(X_t - \mathbf{1}\bar{x}_t)  -\eta Y_t +\eta(\mathbf{I}-\mathbf{W})(X_t-\widetilde{X}_t)\|^2.
    \end{align*}
    Applying \cref{young-1} and using $\|\mathbf{I}-\mathbf{W}_\eta\|\le 2$ and $\|\mathbf{I}-\mathbf{W}\| \le 2$, we obtain
    \begin{align}
        \|X_{t+1}-X_t\|^2 \leq 3\|(\mathbf{W}_\eta -\mathbf{I})(X_t - \mathbf{1}\bar{x}_t)\|^2 + 3\|\eta Y_t\|^2 \nonumber \\ + 3\|\eta(\mathbf{I}-\mathbf{W})(X_t-\widetilde{X}_t)\|^2 \nonumber \\
         \leq 12 \|X_t-\mathbf{1}\bar{x}_t\|^2 + 3\eta ^2 \|Y_t\|^2 + 12 \eta^2 \|X_t-\widetilde{X}_t\|^2. \label{XX1}
    \end{align}
    We next bound $\mathbb{E}[\|Y_t\|^2]$. Using \cref{young-1} and Assumption \ref{smooth}, it holds that 
    \begin{align}
        &\mathbb{E}[\|Y_t\|^2\mid \mathcal{F}_t] \nonumber \\
        &\leq 2\mathbb{E}[\|Y_t-\mathbf{1}\bar{y}_t\|^2] + 2\mathbb{E}[\|\mathbf{1}\bar{y}_t\|^2] \nonumber \\
        &\leq 2\mathbb{E}[\|Y_t-\mathbf{1}\bar{y}_t\|^2] + 4\mathbb{E}[\|\mathbf{1}\bar{y}_t - \mathbf{1}\bar{g}_t\|^2] + 4\mathbb{E}[\|\mathbf{1}\bar{g}_t\|^2] \nonumber \\
        &\leq 2\mathbb{E}[\|Y_t-\mathbf{1}\bar{y}_t \|^2] + 4\mathbb{E}[\|\mathbf{1}\bar{y}_t -\mathbf{1}\bar{g}_t\|^2] \nonumber \\
        &\quad + 8\mathbb{E}[\|\mathbf{1}\bar{g}_t -\mathbf{1}\nabla f_t(\bar{x}_t)\|^2] + 8\mathbb{E}[\|\mathbf{1}\nabla f_t(\bar{x}_t)\|^2] \nonumber \\
        & \leq 2\mathbb{E}[\|Y_t-\mathbf{1}\bar{y}_t \|^2] + 4\mathbb{E}[\|\mathbf{1}\bar{y}_t -\mathbf{1}\bar{g}_t\|^2] \nonumber \\
        &\quad + 8 L^2\mathbb{E}[\|X_t-\mathbf{1}\bar{x}_t \|^2] + 16 nL^2\mathbb{E}[\|\bar{x}_t - x_t^*\|^2]. \label{Le5_1}
    \end{align}
    From Assumption \ref{Estimated variance}, the gradient estimation error is bounded by
    \begin{align}
        &\mathbb{E}[\|\mathbf{1}\bar{y}_t -\mathbf{1}\bar{g}_t\|^2] = n \mathbb{E}[\|\bar{y}_t - \bar{g}_t\|^2] \nonumber \\
        &= \mathbb{E}[\|\frac{1}{n}\sum_{i=1}^n (G_{i,t}-\nabla f_{i,t}(x_{i,t}))\|^2] \nonumber \\
        &\le \frac{1}{n} \sum_{i=1}^n \mathbb{E}[\|G_{i,t}(x_{i,t},u)-\nabla f_{i,t}(x_{i,t})\|^2] \nonumber \\
        &\le v^2_t M_L + \frac{8(m+4)L^2}{n}\|X_t- \mathbf{1}\bar{x}_t\|^2 \nonumber\\
        &\quad + 8(m+4)L^2 \|\bar{x}_t-x^\star_t\|^2 + 4(m+4)\bar{\sigma}_t^2. \label{Le5_2}
    \end{align}
    Plugging \cref{Le5_2} into \cref{Le5_1} yields
    \begin{align}
        &\mathbb{E}[\|Y_t\|^2\mid \mathcal{F}_t] \nonumber \\
        &\leq 2\mathbb{E}[\|Y_t-\mathbf{1}\bar{y}_t \|^2] + 4v_t^2 M_L \nonumber \\
        &\quad +\frac{32(m+4)L^2}{n}\mathbb{E}[\|X_t- \mathbf{1}\bar{x}_t\|^2] \nonumber\\
        &\quad + (32(m+4)L^2 + 16nL^2)\mathbb{E}[\|\bar{x}_t-x^\star_t\|^2] \nonumber\\
        &\quad + 16(m+4)\bar{\sigma}_t^2. \label{XX3}
    \end{align}
    Plugging Eqs. (\ref{XX3}) and (\ref{error_x}) into \cref{XX1} completes the proof.
\end{proof}

\begin{lemma}\label{xx*}
    Let Assumptions \ref{smooth}-\ref{compress} hold. Consider the sequence $\{x_{i,t}\}$  generated by \textbf{OCGT-BF}. Then, for any agent $i\in \mathcal{V}$ and $t<T$,
    \begin{align}
        &\mathbb{E}[\|\bar{x}_{t+1} - x_{t+1}^*\|^2\mid \mathcal{F}_t] \nonumber\\
        &\le (1-\frac{2\eta \mu}{3})\mathbb{E}[\|\bar{x}_t - x_t^*\|^2] \nonumber \\
        &\quad +\frac{\eta L}{4n} \left(\frac{(4m+5)}{\sqrt{m+4}}+\frac{(4m+4)L}{\mu}\right)  \mathbb{E}[\|X_t-\mathbf{1}\bar{x}_t \|^2] \nonumber\\
        &\quad +\left(1+\frac{L\sqrt{m+4}}{36\mu}\right)\|x_{t+1}^*-x_t^*\|^2 \nonumber \\
        &\quad + \left(\frac{1}{L^2}+\frac{ \sqrt{m+4}}{\mu L} \right)(v_t^2 M_L + 4(m+4)\bar{\sigma}_t^2).  
    \end{align}
    \end{lemma}

\begin{proof}
    From \cref{update_bar_x}, one have
    \begin{align}
        &\|\bar{x}_{t+1} - x_{t+1}^*\|^2 =\|\bar{x}_t -\eta \bar{y}_t- x_{t+1}^*\|^2\nonumber\\
        &=\|\bar{x}_t-\eta \nabla f_t(\bar{x}_t)-x_t^*  +\eta (\nabla f_t(\bar{x}_t)- \bar{g}_t) \nonumber \\
        &\quad +\eta (\bar{g}_t-\bar{y}_t) +x_t^* -x_{t+1}^*\|^2. \nonumber
    \end{align}
    Applying \cref{young-2} with $\zeta = \tau > 1$ yields
    \begin{align}
        &\|\bar{x}_{t+1} - x_{t+1}^*\|^2 \nonumber \\
        &\le \tau \|\bar{x}_t-\eta \nabla f_t(\bar{x}_t)-x_t^*\|^2 \nonumber \\
        &\quad + \frac{3\tau}{\tau -1}\Big(\eta^2\|\nabla f_t(\bar{x}_t)-\bar{g}_t\|^2 + \eta^2\|\bar{g}_t-\bar{y}_t\|^2 \nonumber \\
        &\qquad +\|x_t^* -x_{t+1}^*\|^2 \Big).\label{xx1}
    \end{align}
    Taking the conditional expectation $\mathbb{E}[\cdot \mid \mathcal{F}_t]$ and applying Lemmas \ref{Estimated variance}, \ref{sc-lemma}, along with Assumption \ref{smooth},  we obtain
    \begin{align*}
        &\mathbb{E}[\|\bar{x}_{t+1} - x_{t+1}^*\|^2\mid \mathcal{F}_t] \nonumber \\
        &\le \tau (1-\eta \mu)^2\mathbb{E}[\|\bar{x}_t - x_t^*\|^2] \nonumber \\
        &\quad +\frac{3\tau}{\tau -1}\Big(\|x_{t+1}^*-x_t^*\|^2 + \frac{L^2\eta^2}{n}\mathbb{E}[\|X_t-\mathbf{1}\bar{x}_t\|^2] \\
        &\qquad + \eta^2 \mathbb{E}[\|\bar{g}_t-\bar{y}_t\|^2] \Big) \\
        &\le \tau (1-\eta \mu)^2\mathbb{E}[\|\bar{x}_t - x_t^*\|^2] \nonumber \\
        &\quad +\frac{3\tau}{\tau -1}\Big(\|x_{t+1}^*-x_t^*\|^2 + \frac{L^2\eta^2}{n}\|X_t- \mathbf{1}\bar{x}_t\|^2 \nonumber\\
        &\qquad + \eta^2 (v^2_t M_L + \frac{8(m+4)L^2}{n}\|X_t- \mathbf{1}\bar{x}_t\|^2 \\
        &\qquad + 8(m+4)L^2\|\bar{x}_t-x^\star_t\|^2 + 4(m+4)\bar{\sigma}_t^2) \Big).
   \end{align*}
   Choosing $\tau = 1+\frac{\mu}{L\sqrt{m+4}}$ and using the condition $\eta \le \min\{\frac{1}{L},\frac{1}{3\mu},\frac{1}{12L\sqrt{m+4}}\}$ completes the proof.
\end{proof}

\begin{lemma}\label{Xbarx}
    Let Assumptions\ref{graph}-\ref{compress} hold. Consider the sequence $\{x_{i,t}\}_{t=1}^T$ generated by \textbf{OCGT-BF} for any agent $i\in \mathcal{V}$,
    \begin{align}
        &\mathbb{E}[\|X_{t+1}-\mathbf{1}\bar{x}_{t+1}\|^2\mid \mathcal{F}_t] \nonumber \\
        &\quad \le \frac{w_1}{2}\mathbb{E}[\|X_t-\mathbf{1}\bar{x}_t \|^2] +\frac{2w_1\eta^2}{w_2}\mathbb{E}[\|Y_t-\mathbf{1}\bar{y}_t \|^2] \nonumber \\
        &\quad +\frac{8\eta^2 \omega w_1}{w_2}\mathbb{E}[\|X_t-H_t^x\|^2],
     \end{align}
        where $w_1 = 1+ \rho_\eta^2$, $w_2 = 1- \rho_\eta^2$, and $\rho_\eta = \|\mathbf{W}_\eta - \frac{1}{n}\mathbf{1}\mathbf{1}^\top\|$.
\end{lemma}
\begin{proof}
    Recalling Eqs. (\ref{update_X}) and (\ref{update_bar_x}), we have
    \begin{align*}
        &X_{t+1}-\mathbf{1}\bar{x}_{t+1} \\
        &= (\mathbf{W}_\eta X_t - \eta Y_t +\eta(\mathbf{I}-\mathbf{W})(X_t-\widetilde{X}_t)) - \mathbf{1}(\bar{x}_t-\eta\bar{y}_t)\\
        &= (\mathbf{W}_\eta - \frac{1}{n}\mathbf{1}\mathbf{1}^\top)(X_t - \mathbf{1}\bar{x}_t) - \eta (Y_t - \mathbf{1}\bar{y}_t) \\
        &\quad + \eta(\mathbf{I}-\mathbf{W})(X_t-\widetilde{X}_t).
    \end{align*}
    Taking the squared norm and conditional expectation $\mathbb{E}[\cdot \mid \mathcal{F}_t]$, and applying \cref{young-2} with $\zeta = \tau > 1$, it holds that
    \begin{align*}
        &\mathbb{E}[\|X_{t+1}-\mathbf{1}\bar{x}_{t+1}\|^2 \mid \mathcal{F}_t] \\
        & \le \tau \mathbb{E}[\|(\mathbf{W}_\eta - \frac{1}{n}\mathbf{1}\mathbf{1}^\top)(X_t - \mathbf{1}\bar{x}_t)\|^2] \\
        &\quad + \frac{2\tau}{\tau-1} \mathbb{E}[\|-\eta (Y_t - \mathbf{1}\bar{y}_t) + \eta(\mathbf{I}-\mathbf{W})(X_t-\widetilde{X}_t)\|^2] \\
        & \le \tau \rho_\eta^2 \mathbb{E}[\|X_t-\mathbf{1}\bar{x}_t \|^2] + \frac{4\tau \eta^2}{\tau-1}\mathbb{E}[\|Y_t-\mathbf{1}\bar{y}_t \|^2] \\
        &\quad +\frac{4\tau \eta^2}{\tau-1}\mathbb{E}[\|(\mathbf{I}-\mathbf{W})(X_t-\widetilde{X}_t)\|^2] \\
        & \le \tau \rho_\eta^2 \mathbb{E}[\|X_t-\mathbf{1}\bar{x}_t \|^2]  + \frac{4\tau \eta^2}{\tau-1}\mathbb{E}[\|Y_t-\mathbf{1}\bar{y}_t \|^2] \\
        &\quad +\frac{16\tau \eta^2}{\tau-1}\mathbb{E}[\|X_t-\widetilde{X}_t\|^2].
    \end{align*}
    Choosing $\tau=\frac{w_1}{2\rho_\eta^2}$ and using \cref{error_x} complete the proof.
\end{proof}

\begin{lemma}\label{GG}
    Let Assumptions \ref{graph}-\ref{compress} hold. For the sequences $\{x_{i,t}\}$ from \textbf{OCGT-BF}, if $\eta \le \frac{1}{12L\sqrt{m+4}}$, then for any $t<T$,
    \begin{align}
        &\mathbb{E}[\|G_{t+1}-G_t\|^2] \nonumber \\
        &\le (15\left(m+4\right)L^2+1)\mathbb{E}[\|X_t-\mathbf{1}\bar{x}_t\|^2] \nonumber\\
        &\quad +10nL^2(m+4)\mathbb{E}[\|\bar{x}_t-x_t^\star\|^2] \nonumber\\
        &\quad +\left(48+\frac{8(m+4)w_1}{w_2}\right)L^2\eta^2  \mathbb{E} [\|Y_t-\mathbf{1}\bar{y}_t\|^2] \nonumber\\
        &\quad +4\mathbb{E}[\|\nabla F_{t+1}(X_t)-\nabla F_t(X_t)\|^2] \nonumber\\
        &\quad +\frac{30\mu n+L\sqrt{m+4}}{3\mu}(v_t^2 M_L + 4(m+4)\bar{\sigma}_t^2)\nonumber\\
        &\quad +\left(\frac{32w_1(m+4)}{w_2}+48\right)\eta^2 \omega L^2 \mathbb{E}[\|X_t-H_t^x\|^2] \nonumber\\
        &\quad +4(m+4)L^2n\left(1+\frac{L\sqrt{m+4}}{36\mu}\right)  \|x_{t+1}^\star - x_t^\star\|^2.
     \end{align}
    \end{lemma}
\begin{proof}
    Using \cref{young-1}, we have
    \begin{align}
        &\|G_{t+1}-G_t\|^2 \nonumber \\
        &=\|G_{t+1}-\nabla F_{t+1}(X_{t+1})+\nabla F_{t+1}(X_{t+1})-\nabla F_{t+1}(X_t) \nonumber\\
        &\quad +\nabla F_{t+1}(X_t)-\nabla F_t(X_t) + \nabla F_t(X_t) - G_t\|^2 \nonumber\\
        &\le 
        4\|\nabla F_{t+1}(X_t)-\nabla F_t(X_t)\|^2 + 4\|G_t - \nabla F_t(X_t)\|^2 \nonumber \\
        &\quad +4\|G_{t+1}-\nabla F_{t+1}(X_{t+1})\|^2  \nonumber\\
        & \quad + 4\|\nabla F_{t+1}(X_{t+1})-\nabla F_{t+1}(X_t)\|^2.
        \label{GG1}
    \end{align}
   Taking the expectation of \cref{GG1}, using Assumption \ref{smooth} along with Lemma \ref{Estimated variance}, we obtain
    \begin{align*}
        &\mathbb{E}[\|G_{t+1}-G_t\|^2] \\
        &\le 4L^2\mathbb{E}[\|X_{t+1}-X_t\|^2] + 4\mathbb{E}[\|\nabla F_{t+1}(X_t)-\nabla F_t(X_t)\|^2] \\
        &\quad + 4\mathbb{E}[\|G_{t+1}-\nabla F_{t+1}(X_{t+1})\|^2] + 4\mathbb{E}[\|G_t-\nabla F_t(X_t)\|^2] \\
        &\le 4L^2\mathbb{E}[\|X_{t+1}-X_t\|^2] + 4\mathbb{E}[\|\nabla F_{t+1}(X_t)-\nabla F_t(X_t)\|^2] \\
        &\quad + 8 n (v^2_t M_L + \frac{8(m+4) L^2}{n} \|X_t- \mathbf{1}\bar{x}_t\|^2 \\
        &\quad + 8(m+4) L^2 \|\bar{x}_t - x_t^\star\|^2 + 4(m+4)\bar{\sigma}_t^2).
    \end{align*}
    Substituting the bound for $\mathbb{E}[\|X_{t+1}-X_t\|^2]$ from \cref{XX} completes the proof.
\end{proof}

\begin{lemma}\label{lem:Ydiff_bound}
    Let Assumptions \ref{graph}-\ref{compress} hold. For the sequence $\{x_{i,t}\}$ from \textbf{OCGT-BF}, if $\eta \le \frac{1}{12L\sqrt{m+4}}$, then for any $t<T$,
    \begin{align}
        &\mathbb{E}[\|Y_{t+1}-\mathbf{1}\bar{y}_{t+1}\|^2\mid \mathcal{F}_t] \nonumber \\
        &\le \left(\frac{w_1}{2}+\left(\frac{96w_1w_2L^2+16(m+4)w_1^2}{w_2^2}\right)\eta^2\right) \nonumber \\
        &\qquad \times \mathbb{E}[\|Y_t-\mathbf{1}\bar{y}_t \|^2]\nonumber \\
        &\quad +\frac{30(m+4)L^2w_1+2w_1}{w_2} \mathbb{E}[\|X_t-\mathbf{1}\bar{x}_t \|^2] \nonumber \\
        &\quad + \frac{20nL^2(m+4)w_1}{w_2}\mathbb{E}[\|\bar{x}_t - x_t^*\|^2]\nonumber \\
        &\quad + \frac{64w_1^2(m+4)+96w_1w_2}{w_2^2}L^2\eta^2\omega  \mathbb{E}[\|X_t-H_t^x\|^2] \nonumber \\
        &\quad + \frac{8w_1\eta^2\omega}{w_2}\mathbb{E}[\|Y_t-H_t^y\|^2] \nonumber \\
        &\quad + \frac{8w_1}{w_2}\mathbb{E}[\|\nabla F_{t+1}(X_t)-\nabla F_t(X_t)\|^2] \nonumber\\
        &\quad +\frac{60\mu nw_1+2Lw_1\sqrt{m+4}}{3\mu w_2}(v_t^2 M_L + 4(m+4)\bar{\sigma}_t^2)\nonumber\\ 
        &\quad +\frac{8L^2(m+4)nw_1\left(1+\frac{L\sqrt{m+4}}{36\mu}\right)}{w_2}  \|x_{t+1}^\star - x_t^\star\|^2.
    \end{align}
\end{lemma}
\begin{proof}
    Using Eqs. (\ref{update_Y}) and (\ref{update_bar_y}), one has
    \begin{align*}
      &Y_{t+1}-\mathbf{1}\bar{y}_{t+1} \\
      &=(\mathbf{W}_\eta Y_t+G_{t+1}-G_t + \eta(\mathbf{I}-\mathbf{W})(Y_t-\widetilde{Y}_t)) - \mathbf{1}\bar{y}_{t+1}\\
      &=(\mathbf{W}_\eta - \frac{1}{n}\mathbf{1}\mathbf{1}^\top) (Y_t - \mathbf{1}\bar{y}_t) + (I-\frac{1}{n}\mathbf{1}\mathbf{1}^\top)(G_{t+1}-G_t) \\
      &\quad + \eta(\mathbf{I}-\mathbf{W})(Y_t-\widetilde{Y}_t).
    \end{align*}
    Taking the squared norm and conditional expectation $\mathbb{E}[\cdot \mid \mathcal{F}_t]$, applying \cref{young-2} with $\zeta=\tau>1$, using $\|\mathbf{W}_\eta - \frac{1}{n}\mathbf{1}\mathbf{1}^\top\|=\rho_\eta$, $\|I-\frac{1}{n}\mathbf{1}\mathbf{1}^\top\| \le 1$ and $\|\mathbf{I}-\mathbf{W}\| \le 2$, it holds that
    \begin{align*}
        &\mathbb{E}[\|Y_{t+1}-\mathbf{1}\bar{y}_{t+1}\|^2\mid \mathcal{F}_t] \\
        &\le \tau\rho_\eta^2 \mathbb{E}[\|Y_t-\mathbf{1}\bar{y}_t \|^2] + \frac{2\tau}{\tau-1}\mathbb{E}[\|G_{t+1}-G_t\|^2] \\
        &\quad +\frac{2\tau}{\tau-1}\mathbb{E}[\|\eta(\mathbf{I}-\mathbf{W})(Y_t-\widetilde{Y}_t)\|^2] \\
        &\le \tau\rho_\eta^2 \mathbb{E}[\|Y_t-\mathbf{1}\bar{y}_t \|^2] + \frac{2\tau}{\tau-1}\mathbb{E}[\|G_{t+1}-G_t\|^2] \\
        &\quad +\frac{8\tau\eta^2}{\tau-1}\mathbb{E}[\|Y_t-\widetilde{Y}_t\|^2].
    \end{align*}
    Choosing $\tau = \frac{w_1}{2\rho_\eta^2}$ and using \cref{error_y} yields
    \begin{align*}
        &\mathbb{E}[\|Y_{t+1}-\mathbf{1}\bar{y}_{t+1}\|^2\mid \mathcal{F}_t] \\
        &\quad \le \frac{w_1}{2}\mathbb{E}[\|Y_t-\mathbf{1}\bar{y}_t \|^2] + \frac{2w_1}{w_2}\mathbb{E}[\|G_{t+1}-G_t\|^2] \\
        &\qquad + \frac{8w_1\eta^2 \omega}{w_2}\mathbb{E}[\|Y_t-H_t^y\|^2].
    \end{align*}
    Plugging the bound for $\mathbb{E}[\|G_{t+1}-G_t\|^2]$ from \cref{GG} completes the proof.
\end{proof}

\begin{lemma}\label{lem:XdiffH_bound}
    Let Assumptions \ref{smooth}-\ref{compress} hold. Consider the sequence $\{x_{i,t}\}_{t=1}^T$ generated by \textbf{OCGT-BF} for any agent $i\in \mathcal{V}$. If $\eta \le \frac{1}{12L\sqrt{m+4}}$, then for any $t<T$,
        \begin{align}
            &\mathbb{E}[\|X_{t+1}-H_{t+1}^x\|^2\mid \mathcal{F}_t] \nonumber \\
            &\quad \le (b_x+a_x\omega\eta^2)\mathbb{E}[\|X_t-H_t^x\|^2]  +2a_x\mathbb{E}[\|X_t-\mathbf{1}\bar{x}_t \|^2]\nonumber\\
            &\quad +n a_x \eta^2 \mathbb{E}[\|\bar{x}_t - x_t^*\|^2] +a_x \eta^2\mathbb{E}[\|Y_t-\mathbf{1}\bar{y}_t \|^2] \nonumber \\ 
            &\quad + a_x \eta^2 (v_t^2 M_L + 4(m+4)\bar{\sigma}_t^2),
            \label{XH1}
        \end{align}
    where $b_x = \tau_x[1-\alpha_x(1-\omega)]$, $a_x = \frac{12\tau_x}{\tau_x-1}(\frac{1}{12}+L^2(8m+33))$, and $\tau_x>1$. 
\end{lemma}
\begin{proof}
    Recalling \cref{update_hx} and $H_{t+1}^x = H_t^x+\alpha_x Q_t^x$. Thus,
    \begin{align}
        &\|X_{t+1}-H_{t+1}^x\|^2 =\|X_{t+1}-X_t+X_t-H_t^x-\alpha_xQ_t^x\|^2 \nonumber\\
        &= \|X_{t+1}-X_t + (1-\alpha_x)(X_t-H_t^x) \nonumber \\
        &\qquad + \alpha_x(X_t-H_t^x-Q_t^x) \|^2. \nonumber
    \end{align}
    Applying \cref{young-2} with $\zeta=\tau_x > 1$, one gets
    \begin{align*}
         &\mathbb{E}[\|X_{t+1}-H_{t+1}^x\|^2\mid \mathcal{F}_t] \\
         &\le \tau_x \mathbb{E}[\|(1-\alpha_x)(X_t-H_t^x) + \alpha_x(X_t-H_t^x-Q_t^x) \|^2] \\
         &\quad + \frac{\tau_x}{\tau_x-1}\mathbb{E}[\|X_{t+1}-X_t\|^2].
    \end{align*}
    Using convexity of the squared norm and Assumption \ref{compress} yields
    \begin{align*}
         &\mathbb{E}[\|X_{t+1}-H_{t+1}^x\|^2\mid \mathcal{F}_t] \\
         &\le \tau_x ((1-\alpha_x)\mathbb{E}[\|X_t-H_t^x\|^2] + \alpha_x\mathbb{E}[\|X_t-H_t^x-Q_t^x\|^2]) \\
         &\quad + \frac{\tau_x}{\tau_x-1}\mathbb{E}[\|X_{t+1}-X_t\|^2] \\
         &\le \tau_x ( 1 - \alpha_x(1-\omega) ) \mathbb{E}[\|X_t-H_t^x\|^2] \\
         &\quad + \frac{\tau_x}{\tau_x-1}\mathbb{E}[\|X_{t+1}-X_t\|^2].
    \end{align*}
    Substituting the bound for $\mathbb{E}[\|X_{t+1}-X_t\|^2]$ from \cref{XX} and simplifying under the condition $\eta \le \frac{1}{12L\sqrt{m+4}}$ completes the proof with the appropriately defined $a_x$ and $b_x$.
\end{proof}

\begin{lemma}\label{lem:YdiffH_bound2}
    Let Assumptions \ref{smooth}-\ref{compress} hold. For the sequences $\{x_{i,t}\}$ from \textbf{OCGT-BF}, if $\eta \le \frac{1}{12L\sqrt{m+4}}$, then for any $t<T$,
    \begin{align}
        &\mathbb{E}[\|Y_{t+1}-H_{t+1}^y\|^2 \mid \mathcal{F}_t] \nonumber \\
        &\le (b_y + a_y \omega \eta^2)\mathbb{E}[\|Y_t-H_t^y\|^2] \nonumber\\ 
        &\quad + a_y \mathbb{E}[\|Y_t-\mathbf{1}\bar{y}_t \|^2] \nonumber \\ 
        &\quad + a_y \frac{15(m+4)L^2+1}{3} \mathbb{E}[\|X_t-\mathbf{1}\bar{x}_t \|^2] \nonumber\\ 
        &\quad + a_y \frac{10nL^2(m+4)}{3} \mathbb{E}[\|\bar{x}_t - x_t^*\|^2] \nonumber \\ 
        &\quad + a_y (\frac{32w_1(m+4)}{w_2}+48) \frac{\omega \eta^2 L^2}{3} \mathbb{E}[\|X_t-H_t^x\|^2] \nonumber \\ 
        &\quad + a_y \frac{30\mu n+L\sqrt{m+4}}{9\mu} (v_t^2 M_L + 4(m+4)\bar{\sigma}_t^2) \nonumber \\ 
        &\quad + \frac{4a_y}{3} \mathbb{E}[\|\nabla F_{t+1}(X_t)-\nabla F_t(X_t)\|^2] \nonumber \\ 
        &\quad + a_y \frac{4(m+4)L^2n}{3}\left(1+\frac{L\sqrt{m+4}}{36\mu}\right) \|x_{t+1}^\star - x_t^\star\|^2, 
    \end{align}
    where $b_y = \tau_y[1-\alpha_y(1-\omega)]$ and $a_y = \frac{3\tau_y}{\tau_y-1}$ and $\tau_y>1$. 
\end{lemma}
\begin{proof}
    Similar to the proof of \cref{lem:XdiffH_bound}, according to $H_{t+1}^y = H_t^y+\alpha_y Q_t^y$ and  \cref{young-2} with $\zeta=\tau_y > 1$, we obtain
    \begin{align*}
        &\mathbb{E}[\|Y_{t+1}-H_{t+1}^y\|^2 \mid \mathcal{F}_t] \\
        &\quad \le \tau_y [1-\alpha_y(1-\omega)]\mathbb{E}[\|Y_t - H_t^y\|^2] \\
        &\qquad +\frac{\tau_y}{\tau_y-1}\mathbb{E}[\|Y_{t+1}-Y_t\|^2 \mid \mathcal{F}_t].
    \end{align*}
    From \cref{update_Y}, it can be derived that 
    \begin{align*}
        Y_{t+1}-Y_t &=(\mathbf{W}_\eta-\mathbf{I})(Y_t-\mathbf{1}\bar{y}_t ) + (G_{t+1}-G_t) \\
        &\quad + \eta(\mathbf{I}-\mathbf{W})(Y_t-\widetilde{Y}_t).
    \end{align*}
    Taking squared norm and expectation, using \cref{young-1}, $\|\mathbf{W}_\eta-\mathbf{I}\|\le 2$, $\|\mathbf{I}-\mathbf{W}\|\le 2$, and \cref{error_y} yields
    \begin{align*}
        &\mathbb{E}[\|Y_{t+1}-Y_t\|^2] \\
        &\le 3 \mathbb{E}[\|(\mathbf{W}_\eta-\mathbf{I})(Y_t-\mathbf{1}\bar{y}_t )\|^2 ] + 3 \mathbb{E}[\|G_{t+1}-G_t\|^2 ] \\
        &\quad + 3 \mathbb{E}[\|\eta(\mathbf{I}-\mathbf{W})(Y_t-\widetilde{Y}_t)\|^2 ] \\
        &\le 12 \mathbb{E}[\|Y_t-\mathbf{1}\bar{y}_t \|^2] + 3 \mathbb{E}[\|G_{t+1}-G_t\|^2] \\
        &\quad + 12 \eta^2 \mathbb{E}[\|Y_t-\widetilde{Y}_t\|^2] \\
        &\le 12 \mathbb{E}[\|Y_t-\mathbf{1}\bar{y}_t \|^2]+3\mathbb{E}[\|G_{t+1}-G_t\|^2] \\
        &\quad +12\eta^2 \omega \mathbb{E}[\|Y_t-H_t^y\|^2].
    \end{align*}
    Combining these inequalities and substituting the bound for $\mathbb{E}[\|G_{t+1}-G_t\|^2]$ from \cref{GG} leads to the result with the defined $a_y$ and $b_y$.
\end{proof}

\CenteredUppercase{Proof of Lemma \ref{lem:system_ineq}}
\begin{proof}
According to Lemmas \ref{XX}-\ref{lem:YdiffH_bound2}, the matrix $G(\eta) \in \mathbb{R}^{5 \times 5}$ can be constructed as:
\[ 
\resizebox{\columnwidth}{!}{$
G(\eta):= \begin{bmatrix}
    \frac{w_1}{2} & \frac{2w_1\eta^2}{w_2} & 0 & \frac{8\omega w_1\eta^2}{w_2} & 0 \\
    c_1 & \frac{w_1}{2}+c_2 \eta^2 & c_3 & c_4 \eta^2 & \frac{8w_1\omega\eta^2}{w_2} \\
    c_5\eta & 0 & 1-\frac{2\mu\eta}{3} & 0 & 0 \\ 
    2a_x & a_x\eta^2 & na_x & b_x+a_x \omega\eta^2 & 0 \\ 
    c_6 & c_7 & c_8 & c_7 \omega & b_y+a_y\omega\eta
\end{bmatrix}
$}
\]
where the coefficient definitions are 
\begin{align*}
    c_1&:=\frac{30(m+4) L^2 w_1+2 w_1}{w_2}, \\
    c_2 &:= \frac{96 w_1 w_2 L^2+16(m+4) w_1^2}{w_2^2},\\
    c_3 &:= \frac{20 n L^2(m+4) w_1}{w_2}, \\
    c_4&:=\frac{64 w_1^2(m+4)+96 w_1 w_2}{w_2^2} L^2 \omega, \\
    c_5 &:= \frac{L}{4 n}\left(\frac{(4 m+5)}{\sqrt{m+4}}+\frac{(4 m+4) L}{\mu}\right), \\
    c_6&:= \frac{15(m+4) L^2+1}{6} a_y, ~
    c_7 := \frac{(w_1+w_2) a_y}{w_2}, \\
    c_8 &:= \frac{5 n L^2(m+4) a_y}{3}, ~ a_x^\prime := \frac{n a_x }{L^2},~
        a_1 := \frac{8w_1n}{w_2}, \\
        a_2 &:= \frac{8 L^2(m+4) n w_1\left(1+\frac{L \sqrt{m+4}}{36 \mu}\right)}{w_2},\\
        a_3 &:= \frac{60 \mu n w_1+2 L w_1 \sqrt{m+4}}{3 \mu w_2}, \\ 
        a_4 &:= 1+\frac{L \sqrt{m+4}}{36 \mu},~ 
        a_5 := \frac{1}{L^2}+\frac{ \sqrt{m+4}}{\mu L}, \\ 
        a_6 &:= (m+4) L^2 n a_y\left(1+\frac{L \sqrt{m+4}}{36 \mu}\right),\\
        a_7 &:= \frac{(30 \mu n+L \sqrt{m+4}) a_y}{18 \mu}. 
    \end{align*} 

We aim to proof that $0 < \rho(G(\eta)) < 1$ for appropriate stepsize $\eta > 0$. To establish this, it is sufficient to verify that all diagonal entries satisfy $0 < G_{ii}(\eta) < 1$, and that the determinant $\det(I - G(\eta))$ is positive.

\textbf{Step 1: Bounding Diagonal Elements.}

The conditions imposed on each diagonal entry are examined, i.e.,
\begin{align*}
G_{11} &= \frac{w_1}{2} < 1 \quad (\text{since } w_1 = 1+\tilde{\rho}^2 < 2), \\
G_{22} &= \frac{w_1}{2} + c_2 \eta^2 < 1 \quad \Rightarrow \quad
\eta < \sqrt{\frac{w_2}{2c_2}} := \eta_1, \\
G_{33} &= 1 - \frac{3\mu\eta}{2} \in (0,1) \quad \Rightarrow \quad
\eta < \frac{2}{3\mu} := \eta_2, \\
G_{44} &= b_x + 2a_x\omega\eta^2 < 1 \quad \Rightarrow \quad
\eta < \sqrt{\frac{1 - b_x}{2a_x\omega}} := \eta_3, \\
G_{55} &= b_y + a_y\omega\eta < 1 \quad \Rightarrow \quad
\eta < \frac{1 - b_y}{a_y\omega} := \eta_4.
\end{align*}

Define the threshold
\begin{equation}
\eta_{\mathrm{diag}} := \min\{\eta_1, \eta_2, \eta_3, \eta_4\}.
\end{equation}
Then for any $0 < \eta < \eta_{\mathrm{diag}}$, it holds that $0 < G_{ii}(\eta) < 1$ for $i = 1,\dots,5$.

\textbf{Step 2: Bounding the Determinant.}

To analyze the behavior of $\det(I - G(\eta))$, we begin by performing a cofactor expansion along the third row of the matrix $I - G(\eta)$. This yields the decomposition, i.e., 
\begin{align}
  \det(I - G(\eta)) = -c_5 \cdot \det A(\eta) + \frac{3\mu\eta}{2} \cdot \det B(\eta),  \label{eq:det_decomp}
\end{align}
where  $A(\eta)$ and $B(\eta)$ are both $4 \times 4$ matrices whose entries depend continuously on the parameter $\eta \in [0, \eta_{\mathrm{diag}}]$, and are given as follows
\begin{align*}
\resizebox{\columnwidth}{!}{$
    A(\eta) := \begin{bmatrix} \dfrac{-2w_1\eta^2}{w_2} & 0 & \dfrac{-8\omega w_1\eta^2}{w_2} & 0 \\ \dfrac{w_2}{2}-c_2\eta^2 & -c_3 & -c_4\eta^2 & \dfrac{-8\omega w_1\eta^2}{w_2} \\ -a_x\eta^2 & -na_x & 1 - b_x - 2a_x\omega\eta^2 & 0 \\ -c_7 & -c_8 & -c_7\omega & 1 - b_y - a_y\omega\eta \end{bmatrix},
    $}
\end{align*}
\begin{align*}
\resizebox{\columnwidth}{!}{$
    B(\eta) := \begin{bmatrix} \dfrac{w_2}{2} & \dfrac{-2w_1\eta^2}{w_2} & \dfrac{-8\omega w_1\eta^2}{w_2} & 0 \\ -c_1 & \dfrac{w_2}{2} - c_2\eta^2 & -c_4\eta^2 & \dfrac{-8\omega w_1\eta^2}{w_2} \\ -2a_x & -a_x\eta^2 & 1 - b_x - 2a_x\omega\eta^2 & 0 \\ -c_6 & -c_7 & -c_7\omega & 1 - b_y - a_y\omega\eta \end{bmatrix}.
    $}
\end{align*}

We now analyze each of the two terms in the expression for $\det(I - G(\eta)).$

The structure of $A(\eta)$ reveals that every entry in its first row is proportional to $\eta^2$. Therefore, we may factor $\eta^2$ from the entire first row:
\[
\det A(\eta) = \eta^2 \cdot \tilde{D}_A(\eta),
\]
where $\tilde{D}_A(\eta)$ denotes the determinant of a matrix with continuous entries. Since $\tilde{D}_A(\eta)$ is continuous on the compact interval $[0, \eta_{\mathrm{diag}}]$, it is bounded. Define $R_1(\eta) := c_5 \cdot \tilde{D}_A(\eta),$
which is also continuous and bounded on $[0, \eta_{\mathrm{diag}}]$. Thus, the first term in Eq.(\ref{eq:det_decomp}) becomes
\[
-c_5 \cdot \det A(\eta) = -\eta^2 R_1(\eta).
\]
The function $\det B(\eta)$ is continuous since the matrix entries are continuous in $\eta$. Furthermore, we are given that
\[
\det B(0) = D := \frac{w_2^2}{4}(1 - b_x)(1 - b_y) > 0.
\]
By continuity, $\det B(\eta)$ admits the expansion
\[
\det B(\eta) = D + \eta \cdot h(\eta),
\]
where $h(\eta)$ is a continuous function. Substituting into the second term of Eq. (\ref{eq:det_decomp}) yields
\begin{align*}
\frac{3\mu\eta}{2} \cdot \det B(\eta) &= \frac{3\mu\eta}{2} \left( D + \eta h(\eta) \right) \\
&= \eta \cdot \frac{3\mu}{2} D + \eta^2 \cdot \frac{3\mu}{2} h(\eta).
\end{align*}
Define
$R_2(\eta) := \frac{3\mu}{2} h(\eta)$, which is continuous and bounded. Thus,
\[
\frac{3\mu\eta}{2} \cdot \det B(\eta) = \eta \cdot \frac{3\mu}{2} D + \eta^2 R_2(\eta).
\]

Combining both terms, we obtain
\begin{align*}
\det(I - G(\eta)) &= \left( \eta \cdot \frac{3\mu}{2} D + \eta^2 R_2(\eta) \right) - \eta^2 R_1(\eta) \\
&= \eta \cdot \frac{3\mu}{2} D - \eta^2 (R_1(\eta) - R_2(\eta)).
\end{align*}
Define the remainder $R(\eta) := R_1(\eta) - R_2(\eta)$, which is continuous and bounded on $[0, \eta_{\mathrm{diag}}]$. Hence, the final expression becomes
\begin{equation} \label{eq:det_final}
\det(I - G(\eta)) = \eta \left( \frac{3\mu}{2} D \right) - \eta^2 R(\eta),
\end{equation}
where $R(\eta)$ is continuous and bounded  for $\eta \in [0, \eta_{\mathrm{diag}}]$. Let
\[
K := \sup_{0 \le \eta \le \eta_{\mathrm{diag}}} |R(\eta)| < \infty.
\]
Then, for
\[
\eta < \eta_{\mathrm{det}} := \frac{\frac{3\mu}{2} D}{K},
\]
we ensure that $\det(I - G(\eta)) > 0$.

Combining both conditions, define the global threshold
\begin{equation}
\eta^\star := \min\left\{ \eta_{\mathrm{diag}}, \eta_{\mathrm{det}} \right\}.
\end{equation}

Then, for all $0 < \eta < \eta^\star$, it follows that $0 < G_{ii}(\eta) < 1$ and $\det(I - G(\eta)) > 0$, which together imply that $\rho(G(\eta)) < 1$.
\end{proof}
\CenteredUppercase{Proof of Theorem \ref{Th1}}
\begin{proof}
Recalling that the fuction $f_t$ is $L$-smooth,  the descent lemma leads to 
\begin{align}
   f_t(y) \leq f_t(x) + \langle \nabla f_t(x), y - x \rangle + \frac{L_g}{2} \|y - x\|^2. \label{eq:desc}
\end{align}

Let \(y = \bar{x}_t\) and \(x = x_t^\star\). Since \(x_t^\star\) is the minimizer of \(f_t(x)\), the first-order optimality condition under Assumption 2 implies \(\nabla f_t(x_t^\star) = 0\). Substituting these into (\ref{eq:desc}) yields
\[
f_t(\bar{x}_t) \leq f_t(x_t^\star) + \langle 0, \bar{x}_t - x_t^\star \rangle + \frac{L_g}{2} \|\bar{x}_t - x_t^\star\|^2.
\]

Simplifying the above inequality and adding expectations to both sides obtain
\begin{align}
  \mathbb{E}\left[f_t\left(\bar{x}_t\right)\right]-f_t\left(x^{\star}_t\right) \leq \frac{L}{2} \mathbb{E}\left[\left\|\bar{x}_t-x^{\star}_t\right\|^2\right]. \label{theo_1}
\end{align}

Note that $\mathbb{E}\left[\left\|\bar{x}_t-x^{\star}_t\right\|^2\right]$ is the third component of $\mathrm{U}_t$. Using norm equivalence in finite-dimensional vector spaces, there exist positive constants $\lambda_1$ and $\lambda_2$ such that $\|\cdot\| \leq \lambda_1\|\cdot\|_\gamma$ and $\|\cdot\|_\gamma \leq \lambda_2\|\cdot\|$. Also, $\mathbb{E}\left[\left\|\bar{x}_t-x^{\star}_t\right\|^2\right] \le \|\mathrm{U}_t\| \le \lambda_1 \|\mathrm{U}_t\|_\gamma$.
Applying this, we bound expression \cref{theo_1} as
\begin{align}
\mathbb{E}\left[f_t\left(\bar{x}_t\right)\right]-f_t\left(x^{\star}_t\right) \leq \frac{L \lambda_1}{2}\left\|\mathrm{U}_t\right\|_\gamma. \label{Theo_1}
\end{align}
Drawing from \cite[Lemma 5.6.10]{horn2012matrix}, for any $\epsilon > 0$, there exists a matrix norm $\|\cdot\|\gamma$ satisfying 
\begin{align}
    \|G(\eta)\|_\gamma \leq \rho(G(\eta))+\epsilon. \label{rho_gamma}
\end{align}
Since $\rho(G(\eta)) < 1$ for sufficiently small $\eta > 0$, we can choose $\epsilon \in \left(0, 1-\rho(G(\eta))\right)$ and define $\tilde{\rho} := \rho(G(\eta))+\epsilon$. Then, from \cref{rho_gamma}, $\|G(\eta)\|_\gamma \leq \tilde{\rho} < 1$.
Moreover, by \cite[Theorem 5.7.13]{horn2012matrix}, there exists a vector norm $\|\cdot\|\gamma$ compatible with the matrix norm such that $\|M v\|_\gamma \leq\|M\|_\gamma\|v\|_\gamma$ for any matrix $M \in \mathbb{R}^{5 \times 5}$ and $v \in \mathbb{R}^5$.
Applying the norm and triangle inequality to \cref{Lemm14_0}, we obtain
\begin{align*}
    \|\mathrm{U}_{t+1}\|_\gamma &\le \|G(\eta)\mathrm{U}_t + \mathrm{V}_t + v_t^2 \mathbf{d}_1 + \mathbf{d}_2 \bar{\sigma}_t^2 \|_\gamma \\
    &\le \|G(\eta)\|_\gamma \|\mathrm{U}_t\|_\gamma + \|\mathrm{V}_t\|_\gamma + v_t^2 \|\mathbf{d}_1\|_\gamma + \bar{\sigma}_t^2 \|\mathbf{d}_2\|_\gamma \\
    &\le \tilde{\rho} \|\mathrm{U}_t\|_\gamma + \|\mathrm{V}_t\|_\gamma + v_t^2 \|\mathbf{d}_1\|_\gamma + \bar{\sigma}_t^2 \|\mathbf{d}_2\|_\gamma.
\end{align*}
Unrolling the recursion gives:
\begin{align}
    \|\mathrm{U}_t\|_\gamma 
    & \leq \tilde{\rho}^t\|\mathrm{U}_0\|_\gamma + \sum_{k=0}^{t-1} \tilde{\rho}^{t-1-k} (\|\mathrm{V}_k\|_\gamma \nonumber\\
    &\qquad + v_k^2 \|\mathbf{d}_1\|_\gamma + \bar{\sigma}_k^2 \|\mathbf{d}_2\|_\gamma). \label{theom_2} 
 \end{align}
Utilizing the definition of $R_T$ and result \cref{theom_2}, we derive
\begin{align*}
    R_T &\leq \frac{L \lambda_1}{2} \sum_{t=1}^T \big( \tilde{\rho}^{t}\|\mathrm{U}_0\|_\gamma  +\sum_{k=0}^{t-1} \tilde{\rho}^{t-1-k} (\|\mathrm{V}_k\|_\gamma \\
    &\qquad+ v_k^2 \|\mathbf{d}_1\|_\gamma M_L + \bar{\sigma}_k^2 \|\mathbf{d}_2\|_\gamma) \big) \\
        &\le \frac{L \lambda_1}{2} \big( \|\mathrm{U}_0\|_\gamma \sum_{t=1}^T \tilde{\rho}^{t}  +\sum_{t=1}^T \sum_{k=0}^{t-1} \tilde{\rho}^{t-1-k} (\|\mathrm{V}_k\|_\gamma \\
        &\qquad+ v_k^2 \|\mathbf{d}_1\|_\gamma M_L + \bar{\sigma}_k^2 \|\mathbf{d}_2\|_\gamma) \big) \\
        &\le \frac{L \lambda_1}{2} \big( \frac{\|\mathrm{U}_0\|_\gamma}{1-\tilde{\rho}} +\sum_{k=0}^{T-1} (\|\mathrm{V}_k\|_\gamma + v_k^2 \|\mathbf{d}_1\|_\gamma M_L \\
        &\qquad + \bar{\sigma}_k^2 \|\mathbf{d}_2\|_\gamma) \sum_{t=k+1}^T \tilde{\rho}^{t-1-k} \big) \\ 
        &\le \frac{L \lambda_1}{2(1-\tilde{\rho})} \big( \|\mathrm{U}_0\|_\gamma  +\sum_{k=0}^{T-1} (\|\mathrm{V}_k\|_\gamma \\
        &\qquad+ v_k^2 \|\mathbf{d}_1\|_\gamma M_L + \bar{\sigma}_k^2 \|\mathbf{d}_2\|_\gamma) \big) \\ 
        &\le \frac{L \lambda}{2(1-\tilde{\rho})} \Bigg( \|\mathrm{U}_0\|  +\sum_{t=0}^{T-1} (\|\mathrm{V}_t\| \\
        &\qquad+ v_t^2 \|\mathbf{d}_1\|M_L + \bar{\sigma}_t^2 \|\mathbf{d}_2\|) \Bigg)\\
        &\le \frac{L \lambda}{2(1-\tilde{\rho})} \Bigg( \|\mathrm{U}_0\| + C_p \sum_{t=0}^{T-1} \mathbf{p}_t + C_v \sum_{t=0}^{T-1}\mathbf{v}_t \\
        &\qquad + C_d M_L\sum_{t=0}^{T-1} v_t^2   + 4(m+4) C_d \sum_{t=0}^{T-1} \bar{\sigma}_t^2 \Bigg),
    \end{align*}
    where we leverage geometric series properties, swap summation order, use norm equivalence, set $\lambda := \lambda_1 \lambda_2$, and $C_p, C_v$ and $C_d$ are constants defined as
$$ C_p = \sqrt{2a_1^2 + 2(a_y n)^2}, \quad
    C_v = \sqrt{2a_2^2 + a_4^2 + 2a_6^2},$$ and $C_d = \sqrt{a_3^2 + a_5^2 + (a_x^\prime)^2 + a_7^2}$.
    This completes the proof.
\end{proof}

\label{ESGC}
\CenteredUppercase{Proof of Theorem \ref{Th3}}
Following the methodological framework used in the proof of Theorem~\ref{Th1}, we begin by formulating a linear system inequality involving the vector of expectation terms given by
$$
\mathrm{U}_t := \begin{bmatrix}
    \mathbb{E}[\|X_t-\mathbf{1}\bar{x}_t \|^2] \\
    \mathbb{E}[\|Y_t-\mathbf{1}\bar{y}_t \|^2] \\
    \mathbb{E}[\|\bar{x}_t - x_t^*\|^2] \\
    \mathbb{E}[\|X_t-H_t^x\|^2] \\
    \mathbb{E}[\|Y_t-H_t^y\|^2]
\end{bmatrix}. 
$$

According to Eqs. (\ref{update_X}) and (\ref{update_bar_x}) and following the proof of \cref{Xbarx}, the following inequality holds
    \begin{align}
       &\mathbb{E}[\|X_{t+1}-\mathbf{1}\bar{x}_{t+1}\|^2\mid \mathcal{F}_t] \nonumber \\
       &\quad \le \frac{w_1}{2}\mathbb{E}[\|X_t-\mathbf{1}\bar{x}_t \|^2]  +\frac{2w_1\eta^2}{w_2}\mathbb{E}[\|Y_t-\mathbf{1}\bar{y}_t \|^2] \nonumber \\
       &\qquad +\frac{8\eta^2 \omega w_1}{w_2}\mathbb{E}[\|X_t-H_t^x\|^2]. \label{ETS1}
\end{align}

Similarly, using Eqs. (\ref{update_Y}) and (\ref{update_bar_y}), and based on the proof of \cref{lem:Ydiff_bound}, it can be derived that 
    \begin{align}
        &\mathbb{E}[\|Y_{t+1}-\mathbf{1}\bar{y}_{t+1}\|^2\mid \mathcal{F}_t] \nonumber \\
        &\le \frac{w_1}{2}\mathbb{E}[\|Y_t-\mathbf{1}\bar{y}_t \|^2] \nonumber \\
        &\quad + \frac{2w_1}{w_2}\mathbb{E}[\|G_{t+1}-G_t\|^2] \nonumber \\
        &\quad + \frac{8w_1\eta^2 \omega}{w_2}\mathbb{E}[\|Y_t-H_t^y\|^2].
    \end{align}

By applying \cref{GG} to bound $\mathbb{E}[|G_{t+1} - G_t|^2]$ and assuming $\eta < 1/L$, we obtain
   \begin{align}
    &\mathbb{E}[\|Y_{t+1}-\mathbf{1}\bar{y}_{t+1}\|^2\mid \mathcal{F}_t] \nonumber \\
    &\le \left(\frac{w_1}{2}+\frac{96w_1\eta^2L^2}{w_2}\right)\mathbb{E}[\|Y_t-\mathbf{1}\bar{y}_t \|^2] \nonumber \\
    &\quad +\frac{192w_1L^2}{w_2} \mathbb{E}[\|X_t-\mathbf{1}\bar{x}_t \|^2] + \frac{96nw_1}{w_2}\mathbb{E}[\|\bar{x}_t - x_t^*\|^2] \nonumber \\
    &\quad + \frac{96w_1L^2\eta^2 \omega}{w_2}\mathbb{E}[\|X_t-H_t^x\|^2]  + \frac{8w_1\eta^2 \omega}{w_2}\mathbb{E}[\|Y_t-H_t^y\|^2] \nonumber\\ 
    &\quad +\frac{8w_1}{w_2}\|\nabla F_{t+1}(X_t)-\nabla F_t(X_t)\|^2  +\frac{112nw_1}{w_2}\hat{\sigma}^2_t.\label{ETS2}
\end{align}

It follows from \cref{update_bar_x}, similar to the proof of \cref{xx*}, one has
\begin{align}
    &\|\bar{x}_{t+1} - x_{t+1}^*\|^2 \nonumber \\
    &\le \tau \|\bar{x}_t-\eta \nabla f_t(\bar{x}_t)-x_t^*\|^2 + \frac{3\tau}{\tau -1}\Big(\|\eta \nabla f_t(\bar{x}_t)-\eta \bar{g}_t\|^2 \nonumber \\
    &\qquad +\|\eta \bar{g}_t-\eta \bar{y}_t\|^2 +\|x_t^* -x_{t+1}^*\|^2 \Big). 
\end{align}
Taking expectation above, using  \cref{sc-lemma}, Assumptions \ref{smooth} and \ref{unbia}, it holds that
\begin{align}
&\mathbb{E}[\|\bar{x}_{t+1} - x_{t+1}^*\|^2\mid \mathcal{F}_t] \nonumber \\
&\le \tau (1-\eta \mu)^2\mathbb{E}[\|\bar{x}_t - x_t^*\|^2] \nonumber \\
&\quad +\frac{3\tau}{\tau -1}\Big(\frac{L^2\eta^2}{n}\mathbb{E}[\|X_t-\mathbf{1}\bar{x}_t \|^2] \nonumber\\
&\qquad +\frac{\eta^2}{n}\hat{\sigma}^2 +\|x_{t+1}^*-x_t^*\|^2\Big) \nonumber\\ 
&\le \left(1-\frac{3\eta \mu}{2}\right)\mathbb{E}[\|\bar{x}_t - x_t^*\|^2] \nonumber \\
&\quad +\frac{9L^2\eta}{\mu n}\mathbb{E}[\|X_t-\mathbf{1}\bar{x}_t \|^2] \nonumber \\
&\quad +\left(\frac{8}{\mu \eta}+3\right)\|x_{t+1}^*-x_t^*\|^2 + \frac{9\eta}{\mu n}\sigma^2_t, \nonumber \\
&\le \left(1-\frac{3\eta \mu}{2}\right)\mathbb{E}[\|\bar{x}_t - x_t^*\|^2] \nonumber\\
    &\quad +\frac{(8L+1)\eta}{\mu \sqrt{n}}\mathbb{E}[\|X_t-\mathbf{1}\bar{x}_t \|^2] \nonumber \\
    &\quad +\left(\frac{16}{\mu}+1\right)\|x_{t+1}^*-x_t^*\|^2 + \frac{9}{\mu}\hat{\sigma}_t^2, \label{ETS3}
\end{align}
where we choose $\tau = 1+\frac{3\mu \eta}{8}$ and assume $\eta \le \frac{1}{3\mu}$ for the last inequality. 
 
Recalling  \cref{update_hx}, one gets
    \begin{align*}
        H_{t+1}^x & = (1-\alpha_x)H_t^x +\alpha_x \widetilde{X}_t \\
       &= (1-\alpha_x)H_t^x+\alpha_x(Q_t^x+H_t^x)\\
        &=H_t^x+\alpha_x Q_t^x.  
    \end{align*}

Using \cref{XH1} and \cref{young} yields
    \begin{align*}
        &\mathbb{E}[\|X_{t+1}-H_{t+1}^x\|^2\mid \mathcal{F}_t]  \\
        &=\mathbb{E}[\|X_{t+1}-X_t+X_t-H_t^x-\alpha_xQ_t^x\|^2] \\
        &= \mathbb{E}[\|X_{t+1}-X_t+\alpha_x(X_t-H_t^x-Q_t^x) \\
        &\qquad +(1-\alpha_x)(X_t-H_t^x)\|^2]  \\
        & \le \tau_x \mathbb{E}[\|\alpha_x(X_t-H_t^x-Q_t^x)+(1-\alpha_x)(X_t-H_t^x)\|^2] \\
        & \qquad + \frac{\tau_x}{\tau_x-1}\mathbb{E}[\|X_{t+1}-X_t\|^2]. 
    \end{align*}

    By utilizing the convexity of the square norm and assuming $\eta < 1/L$, we obtain 
    \begin{align}
        &\mathbb{E}[\|X_{t+1}-H_{t+1}^x\|^2\mid \mathcal{F}_t] \nonumber \\
        &\le \tau_x \Big(\alpha_x\mathbb{E}[\|X_t-H_t^x-Q_t^x\|^2] \nonumber \\
        &\qquad +(1-\alpha_x)\mathbb{E}[\|X_t-H_t^x\|^2]\Big) \nonumber \\
        &\quad + \frac{\tau_x}{\tau_x-1}\mathbb{E}[\|X_{t+1}-X_t\|^2] \nonumber \\
        &\le \left( \tau_x[1-\alpha_x(1-\omega)]+ \frac{\tau_x}{\tau_x-1}12\eta \omega\right) \nonumber \\ 
        &\qquad \times \mathbb{E}[\|X_t-H_t^x\|^2] \nonumber \\
        &\quad +\frac{\tau_x}{\tau_x-1}\Big[24\mathbb{E}[\|X_t-\mathbf{1}\bar{x}_t \|^2]  +12n\mathbb{E}[\|\bar{x}_t - x_t^*\|^2] \nonumber\\
        &\qquad +12\eta^2 \mathbb{E}[\|Y_t-\mathbf{1}\bar{y}_t \|^2]+12n\eta^2\hat{\sigma}_t^2\Big], 
    \end{align}
where the last inequality is derived from \cref{error_x} and \cref{XX}.

Defining $a_x :=\frac{12\tau_x}{\tau_x-1}>1$ and $b_x := \tau_x[1-\alpha_x(1-\omega)]<1$ gives
\begin{align}
    &\mathbb{E}[\|X_{t+1}-H_{t+1}^x\|^2\mid \mathcal{F}_t] \nonumber \\
    &\le (b_x+a_x \omega \eta)\mathbb{E}[\|X_t-H_t^x\|^2]  +2a_x\mathbb{E}[\|X_t-\mathbf{1}\bar{x}_t \|^2] \nonumber\\
    &\quad +n a_x \mathbb{E}[\|\bar{x}_t - x_t^*\|^2] +a_x \eta^2\mathbb{E}[\|Y_t-\mathbf{1}\bar{y}_t \|^2] \nonumber\\ 
    &\quad + \frac{12n a_x}{L^2} \hat{\sigma}_t^2. \label{ETS4}
\end{align}
    
    Similar to the proof of \cref{ETS4}, one has
    \begin{align}
        &\mathbb{E}[\|Y_{t+1}-H_{t+1}^y\|^2 \mid \mathcal{F}_t] \nonumber \\
        &\quad \le \tau_y [1-\alpha_y(1-\omega)]\mathbb{E}[\|Y_t - H_t^y\|^2] \nonumber \\
        &\qquad +\frac{\tau_y}{\tau_y-1}\mathbb{E}[\|Y_{t+1}-Y_t\|^2].
    \end{align}

It follows from \cref{update_Y} that
\begin{align}
    &\mathbb{E}[\|Y_{t+1}-Y_t\|^2] \nonumber \\
    & \le 6\mathbb{E}[\|Y_t-\mathbf{1}\bar{y}_t \|^2]+3\mathbb{E}[\|G_{t+1}-G_t\|^2] \nonumber \\
    &\quad +6\eta \omega\mathbb{E}[\|Y_t-H_t^y\|^2].
\end{align}

Defining $a_y := \frac{6\tau_y}{\tau_y-1}>1$ and $b_y:=\tau_y[1-\alpha_y(1-\omega)]<1$ yields    
\begin{align}
    &\mathbb{E}[\|Y_{t+1}-H_{t+1}^y\|^2 \mid \mathcal{F}_t] \nonumber \\
    &\le (b_y + \omega\eta a_y)\mathbb{E}[\|Y_t-H_t^y\|^2]  +25a_y\mathbb{E}[\|Y_t-\mathbf{1}\bar{y}_t \|^2] \nonumber \\
    &\quad +48a_yL^2\mathbb{E}[\|X_t-\mathbf{1}\bar{x}_t \|^2] +24n a_y L^2 \mathbb{E}[\|\bar{x}_t - x_t^*\|^2] \nonumber \\
    &\quad +24a_y \omega \eta L^2 \mathbb{E}[\|X_t-H_t^x\|^2]  +28n a_y \hat{\sigma}_t^2\nonumber\\ 
     &\quad +2a_y\|\nabla F_{t+1}(X_t)-\nabla F_t(X_t)\|^2. \label{ETS5}
\end{align}
  
Combining \cref{ETS1,ETS2,ETS3,ETS4,ETS5}, we construct the following system of linear inequality, i.e., 
\begin{align}
    \mathrm{U}_{t+1} \le G^\prime(\eta)\mathrm{U}_t+ \mathrm{V}_t^\prime + \mathbf{d}^\prime \hat{\sigma}_t^2, \nonumber 
\end{align}
where the matrix $G^\prime(\eta)$ and the vectors $\mathrm{V}_t^\prime$ and $\mathbf{d}^\prime$ are defined as
\begin{align*} 
  &G^\prime(\eta):= \\
  &\resizebox{\columnwidth}{!}{$ 
     \begin{bmatrix}
        w_1/2 &c_1^\prime\eta^2 &0 &4c_1^\prime\omega\eta^2 &0 \\
        \frac{192w_1L^2}{w_2} & \frac{w_1}{2}+\frac{96w_1 L^2}{w_2}\eta^2 & \frac{96nw_1}{w_2} & c_2^\prime\omega\eta^2 & 4c_1^\prime\omega\eta^2\\ 
        c_3^\prime\eta &0 &1-3\mu\eta/2 &0 &0\\
        2a_x &a_x\eta^2 &na_x &b_x+a_x\omega\eta &0\\ 
        48a_yL^2 & 25a_y & 24n a_y L^2 & 24a_y \omega \eta L^2 & b_y+a_y\omega\eta 
    \end{bmatrix}
    $}
\end{align*}
\begin{align*}
    \mathrm{V}_t^\prime := \begin{bmatrix} 0 \\ a_1^\prime\mathbf{p}_t^2 \\ a_2^\prime\mathbf{v}_t^2 \\ 0 \\ 2na_y \mathbf{p}_t^2 \end{bmatrix}, 
     \mathbf{d}^\prime := \begin{bmatrix} 0 \\ a_3^\prime \\ 9/\mu \\  12n a_x/L^2 \\ 28na_y \end{bmatrix},
\end{align*}
with coefficient parameters
\begin{align*}
    a_1^\prime &:= \frac{8w_1n}{w_2}, \quad a_2^\prime:=\frac{16}{\mu}+1, \quad a_3^\prime := \frac{112nw_1}{w_2},
\end{align*}
\begin{align*}
    c_1^\prime &= 2w_1/w_2, \quad c_2^\prime = 96w_1L^2/w_2, \\
    c_3^\prime &= \frac{(8L+1)}{\mu \sqrt{n}}, \quad c_4^\prime=24a_yL^2.
\end{align*}

To establish that $0 < \rho(G'(\eta)) < 1$ for an appropriate step size $\eta > 0$, the analysis proceeds by bounding the diagonal elements and the determinant of $I - G'(\eta)$.

The diagonal entries $G'_{ii}(\eta)$ satisfy the following bounds
\begin{align*}
G'_{11} &= \frac{w_1}{2} < 1 \quad (\text{since } w_1 = 1 + \rho^2 < 2), \\
G'_{22} &= \frac{w_1}{2} + \frac{96w_1L^2}{w_2}\eta^2 < 1 \quad \Rightarrow \quad \eta < \sqrt{\frac{w_2}{192w_1L^2}} := \eta_1', \\
G'_{33} &= 1 - \frac{3\mu\eta}{2} \in (0,1) \quad \Rightarrow \quad \eta < \frac{2}{3\mu} := \eta_2', \\
G'_{44} &= b_x + a_x\omega\eta < 1 \quad \Rightarrow \quad \eta < \frac{1 - b_x}{a_x\omega} := \eta_3', \\
G'_{55} &= b_y + a_y\omega\eta < 1 \quad \Rightarrow \quad \eta < \frac{1 - b_y}{a_y\omega} := \eta_4'.
\end{align*}

Define the threshold
\begin{equation}
\eta_{\mathrm{diag}}' := \min\{\eta_1', \eta_2', \eta_3', \eta_4'\}.
\end{equation}
Then for any $0 < \eta < \eta_{\mathrm{diag}}'$, all diagonal elements satisfy $0 < G'_{ii}(\eta) < 1$ for $i = 1,\dots,5$.

To evaluate $I - G'(\eta)$, the approach mirrors that of Lemma~\ref{lem:system_ineq}, using a cofactor expansion along the third row, one has
\begin{align}
\det(I - G'(\eta)) = -c_3^\prime \cdot \det A'(\eta) + \frac{3\mu\eta}{2} \cdot \det B'(\eta), \label{eq:detprime_decomp}
\end{align}
where $A'(\eta)$ and $B'(\eta)$ are $4 \times 4$ matrices constructed by removing the third row and respective columns.

The first row of $A'(\eta)$ is again entirely proportional to $\eta^2$, allowing the factorization by
\[
\det A'(\eta) = \eta^2 \cdot \tilde{D}_{A'}(\eta),
\]
where $\tilde{D}_{A'}(\eta)$ is continuous and bounded. Define $R_1'(\eta) := c_3^\prime \cdot \tilde{D}_{A'}(\eta)$,
so the first term becomes
\[
-c_3^\prime \cdot \det A'(\eta) = -\eta^2 R_1'(\eta).
\]

Similar to the proof of  Lemma \ref{lem:system_ineq}, the function $B'(\eta)$ is continuous , and satisfies  $\det B'(0) = D' > 0$ for some positive constant $D'$. Thus,
\[
\det B'(\eta) = D' + \eta h'(\eta),
\]
where $h'(\eta)$ is continuous. Substituting into the second term yields
\begin{align*}
\frac{3\mu\eta}{2} \cdot \det B'(\eta) 
&= \frac{3\mu\eta}{2}(D' + \eta h'(\eta)) \\
&= \eta \cdot \frac{3\mu}{2} D' + \eta^2 R_2'(\eta),
\end{align*}
where $R_2'(\eta) := \frac{3\mu}{2} h'(\eta)$ is continuous and bounded.

Combining both contributions gives
\begin{align*}
\det(I - G'(\eta)) &= \eta \cdot \frac{3\mu}{2} D' - \eta^2 (R_1'(\eta) - R_2'(\eta)) \\
&= \eta \cdot \frac{3\mu}{2} D' - \eta^2 R'(\eta),
\end{align*}
where $R'(\eta) := R_1'(\eta) - R_2'(\eta)$ is continuous and bounded on $[0, \eta_{\mathrm{diag}}']$.

Let $K' := \sup_{0 \le \eta \le \eta_{\mathrm{diag}}'} |R'(\eta)| < \infty$, and define $\eta_{\mathrm{det}}' := \frac{\frac{3\mu}{2} D'}{K'}$. Then, we have the final threshold
\begin{equation}
\eta^\star := \min\left\{ \eta_{\mathrm{diag}}', \eta_{\mathrm{det}}' \right\}.
\end{equation}
It follows that for all $0 < \eta < \eta^\star$, the conditions $0 < G'_{ii}(\eta) < 1$ and $\det(I - G'(\eta)) > 0$ are satisfied. Consequently, $\rho(G'(\eta)) < 1$.

Similar to the proof of Theorem \ref{Th1}, we establish an upper bound on the dynamic regret $R_T = \sum_{t=1}^T (\mathbb{E}[f_t(\bar{x}_t)] - f_t(x_t^*))$.

We first establish an upper bound for the term $\left\|\mathrm{V}_t^{\prime}\right\|$. From its definition, one has
\begin{align*}
    \|\mathrm{V}_t^\prime\| &= \sqrt{ \left( (a_1^\prime)^2 + 4n^2a_y^2 \right) \|\mathbf{p}_t\|^2 + (a_2^\prime)^2 \|\mathbf{v}_t\|^2 } \\
    &\le \sqrt{ (a_1^\prime)^2 + 4n^2a_y^2 } \mathbf{p}_t + \sqrt{(a_2^\prime)^2} \mathbf{v}_t \\
    &= \sqrt{ (a_1^\prime)^2 + 4n^2a_y^2 } \mathbf{p}_t + |a_2^\prime| \mathbf{v}_t.
\end{align*}

Using the previously derived inequality for the regret $R_T$, , it can be derived that 
\begin{align*}
    R_T &\le \frac{L \lambda}{2(1-\tilde{\rho})} \left( \|\mathrm{U}_0\| + \sum_{t=0}^{T-1} \|\mathrm{V}_t^\prime\| + \sum_{t=0}^{T-1} \|\mathbf{d}^\prime\| \hat{\sigma}_t^2 \right) \\
    &\le \frac{L \lambda}{2(1-\tilde{\rho})} \Bigg( \|\mathrm{U}_0\| + \sqrt{ (a_1^\prime)^2 + 4n^2a_y^2 } \sum_{t=0}^{T-1} \mathbf{p}_t^2 \\
    &\qquad + a_2^\prime \sum_{t=0}^{T-1} \mathbf{v}_t^2 + \|\mathbf{d}^\prime\| \sum_{t=0}^{T-1} \hat{\sigma}_t^2 \Bigg)\\
    &\le \frac{L \lambda}{2(1-\tilde{\rho})} \Bigg( \|\mathrm{U}_0\| + C_p^\prime \sum_{t=0}^{T-1} \mathbf{p}_t^2 + a_2^\prime \sum_{t=0}^{T-1} \mathbf{v}_t^2 \\
    &\qquad + C_d^\prime \sum_{t=0}^{T-1} \hat{\sigma}_t^2 \Bigg),
\end{align*}
where 
\begin{align*}
 C_p^\prime &:= \sqrt{ (a_1^\prime)^2 + 4n^2a_y^2 }, \\
 C_d^\prime & := \sqrt{(a_3^\prime)^2 + \frac{81}{\mu^2} + \frac{144n^2 a_x^2}{L^4} + 784n^2a_y^2}.
\end{align*}
 \qed

\bibliography{ref}
\bibliographystyle{IEEEtran}

\end{document}